\documentclass[10pt,a4paper]{article}
\usepackage[a4paper]{geometry}
\usepackage{amssymb,latexsym,amsmath,amsfonts,amsthm}
\usepackage{graphicx}

\usepackage{epsfig}

\newcommand{\er}{\mathbb{R}}
\newcommand{\cee}{\mathbb{C}}

\newcommand{\zet}{\mathbb{Z}}

\newcommand{\bol}{\hfill\square\\}

\newcommand{\til}{\tilde}
\renewcommand{\Re}{\mathrm{Re}\,}
\renewcommand{\Im}{\mathrm{Im}\,}

\newcommand{\err}{\mathcal{R}}
\newcommand{\ud}{\,\mathrm{d}}

\newtheorem{theorem}{Theorem}[section]
\newtheorem{lemma}[theorem]{Lemma}
\newtheorem{proposition}[theorem]{Proposition}

\newtheorem{corollary}[theorem]{Corollary}

\newtheorem{bvp}[theorem]{Boundary value problem}

\theoremstyle{definition}

\theoremstyle{remark}

\newtheorem{remark}[theorem]{Remark}

\newtheorem{question}[theorem]{Question}

\numberwithin{equation}{section}

\title{A family of Nikishin systems with periodic recurrence coefficients}

\date{\today}

\author{Steven Delvaux\footnotemark[1] , Abey L\'opez\footnotemark[1], Guillermo L\'{o}pez Lagomasino\footnotemark[2]}

\begin{document}

\maketitle
\renewcommand{\thefootnote}{\fnsymbol{footnote}}
\footnotetext[1]{Department of Mathematics, University of Leuven (KU Leuven), Celestijnenlaan 200B, B-3001 Leuven,
Belgium. email: \{steven.delvaux, abey.lopezgarcia\}\symbol{'100}wis.kuleuven.be.} \footnotetext[2]{Departamento de
Matem\'{a}ticas, Universidad Carlos III de Madrid, Avda.~Universidad 30, 28911 Legan\'{e}s, Madrid, Spain. email:
lago\symbol{'100}math.uc3m.es.\\ S.D. and A.L. are Postdoctoral Fellows of the Fund for Scientific Research-Flanders
(FWO), Belgium. G.L.L. is partially supported by research grant MTM 2009-12740-C03-01 of Ministerio de Ciencia e
Innovaci\'{o}n, Spain.}

\begin{abstract}
Suppose we have a Nikishin system of $p$ measures with the $k$th generating measure of the Nikishin system supported on
an interval $\Delta_k\subset\er$ with $\Delta_k\cap\Delta_{k+1}=\emptyset$ for all $k$. It is well known that the
corresponding staircase sequence of multiple orthogonal polynomials satisfies a $(p+2)$-term recurrence relation whose
recurrence coefficients, under appropriate assumptions on the generating measures, have periodic limits of period $p$.
(The limit values depend only on the positions of the intervals $\Delta_k$.) Taking these periodic limit values as the
coefficients of a new $(p+2)$-term recurrence relation, we construct a canonical sequence of monic polynomials
$\{P_{n}\}_{n=0}^{\infty}$, the so-called \emph{Chebyshev-Nikishin polynomials}. We show that the polynomials $P_{n}$
themselves form a sequence of multiple orthogonal polynomials with respect to some Nikishin system of measures, with
the $k$th generating measure being absolutely continuous on $\Delta_{k}$. In this way we generalize a result of the
third author and Rocha \cite{LopRoc} for the case $p=2$. The proof uses the connection with block Toeplitz matrices,
and with a certain Riemann surface of genus zero. We also obtain strong asymptotics and an exact Widom-type formula for
the second kind functions of the Nikishin system for $\{P_{n}\}_{n=0}^{\infty}$.\smallskip


\textbf{2010 AMS classification:} Primary 42C05; Secondary 41A21.
\end{abstract}

\maketitle

\section{Introduction and statement of results}
\label{section:intro}

\subsection{Nikishin system}

Let $p\in\mathbb Z_{>0}$ and let $\Delta_{1},\ldots,\Delta_p\subset\er$ be compact intervals such that
$$ \Delta_k\cap\Delta_{k+1}=\emptyset,\qquad k=1,\ldots,p-1.
$$
Assume that for each $k\in\{1,\ldots,p\}$, $\sigma_k$ is a finite positive measure supported on $\Delta_k$ with density
$\sigma_{k}'(x)>0$ for a.e. $x\in\Delta_{k}$. We denote with
\begin{equation}\label{Niksystem}
\mathcal{M}=(\mu_{1},\ldots,\mu_{p})=\mathcal{N}(\sigma_{1},\ldots,\sigma_{p})
\end{equation}
the Nikishin system generated by the measures $\sigma_k$. The construction of such a system is based on the following
``product operation" for measures; given measures $\sigma_{\alpha}$, $\sigma_{\beta}$, supported on disjoint compact
intervals on $\mathbb{R}$, set
\[
\ud \langle \sigma_{\alpha},\sigma_{\beta} \rangle(x):=\int\frac{\ud\sigma_{\beta}(t)}{x-t}\ud\sigma_{\alpha}(x).
\]
This defines a new measure whose support coincides with that of $\sigma_{\alpha}$. The system $\mathcal{M}$ in
\eqref{Niksystem} is then defined as follows (the notation was introduced in \cite{GRS}):
\[
\mu_{1}:=\sigma_{1},\quad \mu_{2}:=\langle\sigma_{1},\sigma_{2}\rangle, \quad \mu_{3}:=\langle \sigma_{1}, \sigma_{2},
\sigma_{3}\rangle=\langle\sigma_{1},\langle \sigma_{2}, \sigma_{3}\rangle\rangle,\quad \ldots, \quad
\mu_{p}:=\langle\sigma_{1},\langle \sigma_{2},\ldots,\sigma_{p}\rangle\rangle.
\]
Thus the measures $\mu_1,\ldots,\mu_p$ are all of fixed sign and supported on $\Delta_1$. Nikishin systems were
introduced in \cite{Nik}. These systems (and variations of them) have attracted an ever increasing interest during the
last decades, due to their many theoretical and practical applications, see e.g., \cite{Apt,AptKalLopRoc,AKS,
AptLopRoc,AL,BustaLop,CousAssche,DelLop, DrSt,FL,GRS,Kuij,KMW,LopezGarcia,LopRoc,Nik,NS}.

For $n\in\zet_{\geq 0}$ we define the multi-index
\begin{equation}\label{multi:index}
\mathbf{n}:=(n_1,\ldots,n_p):=(\underset{\textrm{$k$ times}}{\underbrace{m+1,\ldots,m+1}},\underset{\textrm{$p-k$
times}}{\underbrace{m,\ldots,m}})\in\mathbb Z^p_{\geq 0},
\end{equation}
where $m\in\zet_{\geq 0}$ and $k\in\{0,1,\ldots,p-1\}$ are such that $n=mp+k$. Note that we have
$|\mathbf{n}|:=n_{1}+n_{2}+\cdots+n_{p}=n$.

Let $(Q_{n})_{n=0}^{\infty}$ be the diagonal sequence of multiple orthogonal polynomials associated with the Nikishin
system $\mathcal{M}$. That is, the polynomial $Q_{n}$ is the monic polynomial of degree $n$ that satisfies for each
$k=1,\ldots,p$ the orthogonality conditions
\begin{equation}\label{Qn:ortho}
\int_{\Delta_{1}}Q_{n}(x)\, x^{l} \ud\mu_{k}(x)=0,\qquad l= 0,1,\ldots,n_{k}-1,
\end{equation}
where $n_{k}$ is the $k$-th component of the multi-index $\mathbf{n}$ in \eqref{multi:index}. The existence and
uniqueness of the sequence $(Q_{n})_{n=0}^{\infty}$ follows from the (weak) perfectness of Nikishin systems, see e.g.\
\cite{NS}. (See also \cite{FL} where it is shown that Nikishin systems are perfect.) It is well known that the
polynomials $Q_n$ satisfy a $(p+2)$-term recurrence relation
\begin{equation}\label{recrelQn}
z Q_{n}(z)=Q_{n+1}(z)+a_{n,n} Q_{n}(z)+a_{n,n-1}Q_{n-1}(z)+\cdots+a_{n,n-p}\,Q_{n-p}(z),
\end{equation}
with initial conditions
\[
Q_{-p}\equiv Q_{-p+1}\equiv \cdots\equiv Q_{-1}\equiv 0, \quad Q_{0}\equiv 1,
\]
and with $a_{n,n-p}\neq 0$ for all $n\geq p$. The recurrence coefficients $a_{n,m}$ in \eqref{recrelQn} can be viewed
as the entries of the banded Hessenberg operator
\begin{equation}\label{matrixA}
A=\begin{pmatrix} a_{0,0} & 1  & &    \\ a_{1,0} & a_{1,1} & 1 &   \\ \vdots  & \ddots & \ddots &  \ddots   \\ a_{p,0}
& a_{p,1} & \ddots & \ddots   \\
 & a_{p+1,1} & a_{p+1,2} & \ddots  \\
 &  & \ddots & \ddots
\end{pmatrix}.
\end{equation}
Thus $A$ is a semi-infinite matrix with one superdiagonal, filled with $1$'s, and with $p$ subdiagonals. All the other
entries of $A$ are equal to zero.

\subsection{Chebyshev-Nikishin polynomials}

It was proved in \cite{AptLopRoc} that the following ratio asymptotic formulas hold:
\begin{equation}\label{ratioasymp}
\lim_{m\rightarrow\infty}\frac{Q_{mp+k}(z)}{Q_{mp+k-1}(z)}=:F_{k}(z),\qquad k=1\ldots,p,
\end{equation}
where the limits are valid uniformly on compact subsets of $\mathbb{C}\setminus \Delta_{1}$. The limiting ratios
$F_k(z)$ will be described in more detail in Section~\ref{section:ratioasy}.

The ratio asymptotics \eqref{ratioasymp} imply, see \cite{AptKalLopRoc}, that for each fixed $(i,j)$, $0\leq j\leq
p-1$, $j\leq i\leq j+p$ the following limit exists:
\begin{equation}\label{def:limitreccoeff}
\lim_{m\rightarrow\infty} a_{mp+i,mp+j}=:b_{i,j}.
\end{equation}
For all $m \in \mathbb{Z}_{\geq 0}$, we take $b_{mp+i,mp+j}:= b_{i,j}$. This means that the banded Hessenberg operator
$A$ in \eqref{matrixA} is a compact perturbation of the tridiagonal block Toeplitz operator
\begin{equation}\label{defH}
T=\begin{pmatrix} B_{0} & B_{-1}\\ B_{1} & B_{0} & B_{-1} \\
 & B_{1} & B_{0} & B_{-1} \\
 & & B_{1} & B_{0} & \ddots \\
 & & & \ddots & \ddots
\end{pmatrix},
\end{equation}
where the blocks $B_{k}$ are of size $p\times p$ and given by
\begin{equation}\label{blocks:B01}
B_{0}=\begin{pmatrix} b_{0,0} & 1 \\ b_{1,0} & b_{1,1} & \ddots \\ \vdots & \vdots & \ddots & 1 \\ b_{p-1,0} &
b_{p-1,1} & \ldots &  b_{p-1,p-1}
\end{pmatrix},\quad
B_{1}=\begin{pmatrix} b_{p,0} & b_{p,1} & \ldots & b_{p,p-1} \\
 & b_{p+1,1} & \ldots & b_{p+1,p-1}\\
 &  & \ddots & \vdots \\
 & & & b_{2p-1,p-1}
\end{pmatrix},
\end{equation}
\begin{equation}\label{blocks:Bminus1} B_{-1}=\begin{pmatrix}
0 & 0 & \ldots & 0 \\ \vdots & \vdots & & \vdots \\ 0 & 0 & \ldots & 0 \\ 1 & 0 & \ldots & 0 \\
\end{pmatrix}.
\end{equation}
We view the diagonals of $T$ as infinite periodic sequences with period $p$.

\begin{theorem}\label{theorem:relations:bij}
The following relations hold for all $k,j$ and all $l \in\{1,\ldots,p-1\}$:
\begin{align}
& b_{k,k-1}=b_{j,j-1}=:\beta,\label{rel:b:1}\\ & b_{k+1,k- l}-b_{k,k- l-1}= b_{k,k- l}(b_{k- l,k- l} -b_{k ,k }).
\label{rel:b:2}
\end{align}
\end{theorem}

Theorem~\ref{theorem:relations:bij} generalizes a result in \cite{LopRoc} for the case $p=2$. It will be proved in
Section~\ref{section:rellimitcoeff}. However, Theorem~\ref{theorem:relations:bij} will not be used in the rest of the
paper.

We define the \emph{Chebyshev-Nikishin polynomials} \cite{LopRoc} as the monic polynomials $(P_{n})_{n=0}^{\infty}$
associated with the operator $T$ in \eqref{defH}; that is, \begin{equation}\label{Pn:def} P_n(z) = \det(zI_n-T_n),
\end{equation}
where $I_n$ denotes the identity matrix of size $n$, and $T_n$ is the $n\times n$ principal submatrix of $T$. The
polynomials $P_n$ satisfy the recurrence relation
\[ zP_n(z) = P_{n+1}(z) + b_{n,n}P_n(z) + \cdots + b_{n,n-p} P_{n-p}(z)
\]
with the initial conditions
\[
P_{-p}\equiv P_{-p+1}\equiv \cdots\equiv P_{-1}\equiv 0, \quad P_{0}\equiv 1.
\]
Note that for $p=1$ and $\Delta_{1}=[-2,2]$, the polynomials $P_{n}$ reduce to the classical Chebyshev polynomials of
the second kind for this interval. This explains the name Chebyshev-Nikishin polynomials \cite{LopRoc}.

\subsection{Nikishin system generated by the Chebyshev-Nikishin polynomials}

It is easy to see that there exist positive measures $(\nu_1,\ldots,\nu_p)$ supported on $\Delta_1$ so that
\begin{equation}\label{cauchy1} \frac{1}{F_{k}(z)} = \int_{\Delta_1} \frac{1}{z-t} \ud\nu_k(t),\qquad k=1,\ldots,p.
\end{equation}
This is a consequence of the ratio asymptotics in \eqref{ratioasymp} and the interlacing of the zeros of $Q_{mp+k}(z)$
and $Q_{mp+k-1}(z)$ \cite{AptLopRoc}. Alternatively, the integral representation \eqref{cauchy1} can be obtained from
the analytic properties of $\frac{1}{F_{k}(z)}$ using \cite[Theorem A.6]{KN} (see \eqref{ratio:asy:psi} below) which
ensures that the measures $\nu_k, k=1,\ldots,p,$ are absolutely continuous with respect to the Lebesgue measure.

The main goal of this paper is to establish the following two theorems on the Chebyshev-Nikishin polynomials $P_n(z)$.

\begin{theorem}\label{theorem:main:ortho}
Let $(\nu_1,\ldots,\nu_p)$ be the measures on $\Delta_1$ defined in \eqref{cauchy1}. Then the Chebyshev-Nikishin
polynomials $P_n$ satisfy the orthogonality conditions
\begin{equation}\label{Pn:mop}
\int_{\Delta_1} P_{n}(x)\, x^{l} \ud\nu_k(x)=0,\qquad l=0,\ldots,n_{k}-1,\quad k=1,\ldots,p,
\end{equation}
where $n_{k}$ is the $k$-th component of the multi-index $\mathbf{n}$ in \eqref{multi:index}.
\end{theorem}

\begin{theorem} \label{theorem:main:Nikishin}
The measures $(\nu_1,\ldots,\nu_p)$ form a Nikishin system on $(\Delta_1,\ldots,\Delta_p)$. That is, there exist
measures $\rho_k$ supported on $\Delta_k$, $k=1,\ldots,p$, such that
$(\nu_{1},\ldots,\nu_{p})=\mathcal{N}(\rho_{1},\ldots,\rho_{p})$. The measure $\rho_k$ is absolutely continuous on
$\Delta_k$, with its density $\rho_k'(x)$ being non-vanishing and of a fixed sign (either positive or negative) on the
interior of $\Delta_k$.
\end{theorem}

Theorems~\ref{theorem:main:ortho} and \ref{theorem:main:Nikishin} will be proved in
Sections~\ref{section:proof:main:ortho} and~\ref{section:proof:main:Nik} respectively. For $p=2$ these theorems reduce
to a result in \cite{LopRoc}. Our proofs will be markedly different from the ones in \cite{LopRoc}. A generalization of
Theorem~\ref{theorem:main:ortho} to other systems of multi-indices $\mathbf n$ will be given in
Section~\ref{section:concluding:remark}.

\subsection{About the proof of Theorem~\ref{theorem:main:ortho}: ratio asymptotics}
\label{subsection:aboutproof}

Let
\begin{equation}\label{symbol:Toeplitz}
F(z,x):=\frac{1}{z}\, B_{-1}+B_{0}+z B_{1}-x I_{p}
\end{equation}
be the block Toeplitz symbol associated with $T$ in \eqref{defH}, and let
\begin{equation}\label{algebraic:equation}
f(z,x):=\det F(z,x).\end{equation} The algebraic equation $f(z,x)=0$ has $p+1$ roots $z_{1}(x),\ldots,z_{p+1}(x)$
(counting multiplicities) that we label so that
\begin{equation}\label{ordering:roots}
|z_{1}(x)|\leq |z_{2}(x)|\leq \cdots\leq |z_{p+1}(x)|,\qquad x\in\mathbb{C}.
\end{equation}
We introduce the sets
\begin{equation}\label{defGammas}
\Gamma_{k}:=\{x\in\mathbb{C}: |z_{k}(x)|=|z_{k+1}(x)|\},\qquad k=1,\ldots,p.
\end{equation}

In the proof of Theorem~\ref{theorem:main:ortho}, the following result will play a pivotal role.

\begin{proposition}\label{prop:QP:samelimit}
We have that
\begin{equation}\label{GammaDelta}
\Gamma_k=\Delta_k,\qquad k=1,\ldots,p,
\end{equation}
and \begin{equation}\label{ratio:asy:PQ} \lim_{m\rightarrow\infty}\frac{P_{mp+k}(x)}{P_{mp+k-1}(x)}=
\lim_{m\rightarrow\infty}\frac{Q_{mp+k}(x)}{Q_{mp+k-1}(x)}=: F_{k}(x),\qquad k=1,\ldots,p,
\end{equation}
uniformly on each compact subset of  $\cee\setminus(\Delta_1\cup\mathcal A)$, with $\mathcal A$ a set of finite
cardinality. Here we use the notation in \eqref{ratioasymp}.
\end{proposition}

Proposition~\ref{prop:QP:samelimit} will be proved in Section~\ref{section:ratioasy}. The equality \eqref{GammaDelta}
was already obtained by Aptekarev \cite[Prop.~2.1--2.2]{Apt} but we will provide an alternative proof.

\begin{remark} Once we have Theorems~\ref{theorem:main:ortho} and \ref{theorem:main:Nikishin} to our disposal,
it will follow that the polynomials $P_n$ are of multiple orthogonality with respect to a Nikishin system. Equation
\eqref{ratio:asy:PQ} is then a consequence of the main result in \cite{AptLopRoc}. It also follows that the zeros of
$P_n$ are simple and lie in the interior of $\Delta_1$, with the zeros of two consecutive polynomials interlacing
\cite[Theorem 2.1]{AptLopRoc}. Therefore in \eqref{ratio:asy:PQ} the convergence is uniform on each compact subset of
$\mathbb{C} \setminus \Delta_1$.
\end{remark}

\subsection{Strong asymptotics, Widom-type formulas, and the second kind functions}

As a consequence of Widom's determinant identity for block Toeplitz matrices \cite[Section 6]{Widom1}, we have the
following \emph{exact} formula for $P_{mp+k}(x)$ (see \cite[Section 4]{DelLop})
\begin{equation}\label{Widom:Pn}
P_{mp+k}(x) = \frac{(-1)^{p+k}}{f_{p}}\sum_{j=1}^{p+1} \frac{\det F^{[p,k+1]}(z_j(x),x)}{\prod_{i=1,i\neq
j}^{p+1}(z_{j}(x)-z_{i}(x))} z_j(x)^{-m-1},
\end{equation}
for all $m\in\zet_{\geq 0}$ and $k=0,1,\ldots,p-1$, where we write
\begin{equation}\label{def:fp}
f_{p}=\prod_{i=0}^{p-1} b_{p+i,i},
\end{equation}
and where we denote with $F^{[p,k+1]}$ the submatrix of $F$ in \eqref{symbol:Toeplitz} obtained by deleting row $p$ and
column $k+1$. Note that we label the rows and columns of $F$ as $1,2,\ldots,p$, i.e., we start counting from $1$ rather
than~$0$. In writing \eqref{Widom:Pn}, we are implicitly using the fact that $f_{p}\neq 0$. This will be justified in
Section~\ref{section:ratioasy}, see \eqref{equation:psi:k}.

Equation~\eqref{Widom:Pn} holds for all $x$ for which the roots $z_j(x)$ are pairwise distinct. If $x\in\cee$ is such
that two roots $z_j(x)$ and $z_k(x)$ with $j\neq k$ are equal (there are only finitely many such $x$) then
\eqref{Widom:Pn} remains true provided we replace the right hand side by its limiting value.

As a consequence of \eqref{Widom:Pn} we have the strong asymptotic formula
\begin{equation}\label{strongasy:Pn}
\lim_{m\rightarrow\infty}P_{mp+k}(x)\, z_{1}(x)^{m+1}=\frac{(-1)^{p+k}}{f_{p}}\frac{\det
F^{[p,k+1]}(z_{1}(x),x)}{\prod_{i=2}^{p+1}(z_{1}(x)-z_{i}(x))},
\end{equation}
for any fixed $k\in\{0,\ldots,p-1\}$, uniformly on compact subsets of $\mathbb{C}\setminus\Gamma_{1}$ \cite[Section
4]{DelLop}.

\smallskip
In this paper we will obtain similar Widom-type formulas for the second kind functions of the Nikishin system generated
by the $P_n$. Recall that the second kind functions $\Psi_{n,l}(z)$ are defined recursively by \cite{GRS}
\begin{equation}\label{secondkind:1} \Psi_{n,0}(z) := P_n(z)
\end{equation}
and \begin{equation}\label{secondkind:2} \Psi_{n,l}(z) = \int_{\Delta_l} \frac{\Psi_{n,l-1}(t)}{z-t}\ud \rho_l(t),
\end{equation}
for $l=1,\ldots,p$, where $\rho_l$ is the measure supported on $\Delta_l$ described in
Theorem~\ref{theorem:main:Nikishin}. The statement of the Widom-type formulas for $\Psi_{n,l}(z)$ requires some extra
notations and is deferred to Section~\ref{section:Widom:secondkind}.

\subsection{Outline of the paper}

This paper is organized as follows. In Section~\ref{section:kalyagin} we recall some general results on multiple
orthogonality relations for banded Hessenberg operators. In Section~\ref{section:ChebNikpoly} we apply these results to
the Chebyshev-Nikishin polynomials $P_n$. Section~\ref{section:ratioasy} proves Proposition~\ref{prop:QP:samelimit} on
the ratio asymptotics of the polynomials $P_n$. In Section~\ref{section:proof:main:ortho} we use these considerations
to prove Theorem~\ref{theorem:main:ortho}. In Section~\ref{section:proof:main:Nik} we prove
Theorem~\ref{theorem:main:Nikishin}. In Section~\ref{section:rellimitcoeff} we prove
Theorem~\ref{theorem:relations:bij} on the relations between the coefficients $b_{i,j}$ in \eqref{def:limitreccoeff}.
In Section~\ref{section:Widom:secondkind} we obtain strong asymptotics and exact Widom-type formulas for the second
kind functions $\Psi_{n,l}(z)$ in \eqref{secondkind:1}--\eqref{secondkind:2}. Finally,
Section~\ref{section:concluding:remark} discusses the generalization of our results to some alternative systems of
multi-indices $\mathbf{n}$.

\section{Multiple orthogonality relations for banded Hessenberg operators}
\label{section:kalyagin}

In this section we recall some general results on multiple orthogonality relations for banded Hessenberg matrices,
following Kaliaguine \cite{Kal1}, see also Van Iseghem \cite{VIse} and Kaliaguine \cite{Kal2}. We consider a banded Hessenberg matrix $A$ with $p+2$
diagonals of the form \eqref{matrixA}. We associate to $A$ the following sequences of polynomials that we denote by
$(q_{n})_{n}$ and $(p_{n}^{(j)})_{n}$, where $j\in\{1,\ldots,p\}$. We require these polynomials to satisfy the
recurrence relation
\begin{equation}\label{recqnpn}
z y_{n}=y_{n+1}+a_{n,n}\, y_{n}+a_{n,n-1}\, y_{n-1}+\cdots+a_{n,n-p}\, y_{n-p},\qquad n\geq 0,
\end{equation}
with initial conditions
\begin{equation}\label{initcond}
\begin{array}{cccccccc}
n & = & -p & -p+1 & -p+2 & \ldots & -1 & 0\\ \hline p_{n}^{(1)} & = & 1 & 0 & 0 & \ldots & 0 & 0 \\ p_{n}^{(2)} & = & 0
& 1 & 0 & \ldots & 0 & 0 \\ p_{n}^{(3)} & = & 0 & 0 & 1 & \ldots & 0 & 0 \\
 & \vdots & & & & \vdots & &\\
p_{n}^{(p)} & = & 0 & 0 & 0 & \ldots & 1 & 0 \\ q_{n} & = & 0 & 0 & 0 & \ldots & 0 & 1
\end{array}
\end{equation}
In our definition, we take $a_{i,j}=-1$ for $0\leq i\leq p-1$ and $j\leq -1$. Observe that for $n\geq 0$, $q_{n}$ is a
polynomial of degree $n$ and $p_{n}^{(j)}$ is a polynomial of degree $n-1$ for all $j\in\{1,\ldots,p\}$.

Let $A_{n}$ be the truncation of $A$ to the first $n$ rows and columns. We define
\begin{align}
D_{n}(z) & :=\det (z I_{n}-A_{n}),\label{defDn}\\ D_{n}^{(j)}(z) & :=\det (z
I_{n-j}-\left(A_{n}\right)^{[1,\ldots,j;1,\ldots,j]}),\qquad 1\leq j\leq p, \label{defnDnj}
\end{align}
where $I_{k}$ denotes the identity matrix of size $k$, and $M^{[1,\ldots,j;1,\ldots,j]}$ denotes the submatrix of $M$
obtained after deleting the first $j$ rows and columns. Of course, we set
\[
D_{0}(z)\equiv D_{1}^{(1)}(z)\equiv D_{2}^{(2)}(z)\equiv \cdots \equiv D_{p}^{(p)}(z)\equiv 1,\\
\]
and we take $D_{n}(z)\equiv 0$, $n<0$, and $D_{n}^{(j)}(z)\equiv 0$, $n<j$.

Expanding $D_n$ by its last row, it is easy to check that
$$
q_{n}=D_{n},\qquad n\geq 0
$$
since both sequences satisfy the same recurrence relation and initial conditions. We also have
\begin{lemma}\label{lemmarel}
The following relation holds for every $j\in\{1,\ldots,p\}$:
\begin{equation}\label{relpnDn}
p_{n}^{(j)}=D_{n}^{(1)}+D_{n}^{(2)}+\cdots+D_{n}^{(j)},\qquad n\geq 0.
\end{equation}
\end{lemma}
\begin{proof}
It easily follows from the definition \eqref{defnDnj} that the sequence of polynomials $D_{n}^{(j)}$ satisfies the
recurrence relation
\begin{equation}\label{recDnj}
z D_{n}^{(j)}(z)=D_{n+1}^{(j)}(z)+a_{n,n}\,
D_{n}^{(j)}(z)+a_{n,n-1}\,D_{n-1}^{(j)}(z)+\cdots+a_{n,n-p}\,D_{n-p}^{(j)}(z),\qquad n\geq j.
\end{equation}
Since $p_{1}^{(1)}(z)=1=D_{1}^{(1)}(z)$, and $p_{n}^{(1)}=D_{n}^{(1)}=0$ for $1-p\leq n\leq 0$, it follows from
\eqref{recqnpn} and \eqref{recDnj} that $p_{n}^{(1)}=D_{n}^{(1)}$ for all $n\geq 0$.

Next we show that
\begin{equation}\label{rel}
p_{n}^{(j)}(z)-p_{n}^{(j-1)}(z)=D_{n}^{(j)}(z) \qquad\mbox{for all}\,\,n\geq 0.
\end{equation}
Evidently,
\begin{equation}\label{recdifpn}
z(p_{n}^{(j)}-p_{n}^{(j-1)})=p_{n+1}^{(j)}-p_{n+1}^{(j-1)}+a_{n,n}(p_{n}^{(j)}-p_{n}^{(j-1)})
+\cdots+a_{n,n-p}(p_{n-p}^{(j)}-p_{n-p}^{(j-1)}),\qquad n\geq 0.
\end{equation}
Recall that the initial conditions \eqref{initcond} hold, in particular $p_{j-1-p}^{(j)}=1$ and $p_{j-2-p}^{(j-1)}=1$.
By definition, $a_{k,j}=-1$ for $0\leq k\leq p-1$ and $j\leq -1$, so if we set $n=0$ in \eqref{recdifpn} we get
$p_{1}^{(j)}-p_{1}^{(j-1)}=0$. If we continue evaluating $n$ from $1$ to $j-2$ in \eqref{recdifpn}, we obtain
$p_{n}^{(j)}-p_{n}^{(j-1)}=0$ for $2\leq n\leq j-1$. Therefore \eqref{rel} holds for $0\leq n\leq j-1$.

If we now put $n=j-1$ in \eqref{recdifpn}, we get $p_{j}^{(j)}-p_{j}^{(j-1)}=1=D_{j}^{(j)}$. And now \eqref{rel} will
follow immediately from \eqref{recdifpn}, \eqref{recDnj}, and the fact that $p_{n}^{(j)}-p_{n}^{(j-1)}=D_{n}^{(j)}$ for
$j-p\leq n\leq j$. Therefore \eqref{relpnDn} can be proved now by induction on $j$.
\end{proof}

We recall some results due to Kaliaguine \cite{Kal1,Kal2}, see also Van Iseghem \cite{VIse}. In order to state these results,
we introduce some notation. We will assume that the coefficients of the matrix $A$ are uniformly bounded. Let
$\{\mathbf e_{n}\}_{n=0}^{\infty}$ denote the standard basis in $l^{2}$, and consider the following resolvent
functions
\begin{equation}\label{resolvf}
g_{j}(z):=(R_{z} \mathbf e_{j-1},\mathbf e_{0}),\qquad 1\leq j\leq p,
\end{equation}
where $R_{z}=(zI-A)^{-1}$ is the resolvent operator, and $(\cdot,\cdot)$ is the standard inner product in $l^{2}$. We
also define
\begin{equation}
\phi_{j}(z):=g_{1}(z)+g_{2}(z)+\cdots+g_{j}(z),\qquad 1\leq j\leq p.
\end{equation}
For each $j=1,\ldots,p$, we introduce a linear functional $L_{j}$ defined on the space of polynomials by
\[
L_{j}(z^{n})=(A^{n} \mathbf v_{j},\mathbf e_{0}),
\]
where $\mathbf v_{j}:=\mathbf e_{0}+\cdots+\mathbf e_{j-1}$. Observe that the sequence
$(\mu_{j,n})_{n=0}^{\infty}=(L_{j}(z^n))_{n=0}^{\infty}$ is the sequence of ``moments" for $\phi_{j}$, in the sense
that
\[
\phi_{j}(z)=\sum_{n=0}^{\infty}\frac{\mu_{j,n}}{z^{n+1}},
\]
for all $z\in\cee$ sufficiently large.

\begin{theorem}[See \cite{Kal1,VIse}]\label{theoKal}
$a)$ The vector of rational functions
\[
\Big(\frac{p_{n}^{(1)}}{q_{n}},\frac{p_{n}^{(2)}}{q_{n}},\ldots,\frac{p_{n}^{(p)}}{q_{n}}\Big)
\]
is a Hermite-Pad\'{e} approximant to the system $(\phi_{1},\phi_{2},\ldots,\phi_{p})$, with respect to the multi-index
$\mathbf{n}$ in \eqref{multi:index}. That is, for each $j=1,\ldots,p$, we have
\[
q_{n}(z) \phi_{j}(z)-p_{n}^{(j)}(z)=O\Big(\frac{1}{z^{n_{j}+1}}\Big),\qquad z\rightarrow\infty,
\]
where $n_{j}$ is the $j$-th component of $\mathbf{n}$.

$b)$ For each $j\in\{1,\ldots,p\}$, the polynomial $q_{n}$ satisfies the following orthogonality conditions:
\[
L_{j}(q_{n} z^{l})=0,\qquad l=0,\ldots,n_{j}-1.
\]
\end{theorem}

\begin{remark}\label{rmkHP}
Making use of Theorem \ref{theoKal} and Lemma \ref{lemmarel}, by linearity we obtain that
\[
\Big(\frac{D_{n}^{(1)}}{D_{n}},\frac{D_{n}^{(1)}+c_{1,1}
D_{n}^{(2)}}{D_{n}},\ldots,\frac{D_{n}^{(1)}+c_{1,p-1}D_{n}^{(2)}+c_{2,p-1}D_{n}^{(3)}+\cdots+c_{p-1,p-1}D_{n}^{(p)}}{D_{n}}\Big)
\]
is a Hermite-Pad\'{e} approximant to the system of functions
\[
(g_{1},g_{1}+c_{1,1}g_{2},g_{1}+c_{1,2}g_{2}+c_{2,2}g_{3},\ldots,g_{1}+c_{1,p-1}g_{2}+\cdots+c_{p-1,p-1}g_{p}),
\]
with respect to the multi-index \eqref{multi:index}. Here, $c_{i,j}$ denote arbitrary constants.
\end{remark}

The recent survey \cite{AptKuij} contains information on the latest developments in the theory of Hermite-Pad\'e approximation, multiple orthogonal polynomials, and their applications to random matrix theory.

\section{Chebyshev-Nikishin polynomials}
\label{section:ChebNikpoly}

In this section we  apply the results of Section~\ref{section:kalyagin} to the Chebyshev-Nikishin polynomials
$(P_{n})_{n=0}^{\infty}$ defined by \eqref{def:limitreccoeff}--\eqref{Pn:def}. To this end it is convenient to use an
alternative determinantal formula for $P_n$. In fact, if we applied the results of Section~\ref{section:kalyagin}
directly to the operator $T$ in \eqref{defH} then we would need to control the limiting ratios
$$ \frac{\det \big(z
I_{n-j}-\left(T_{n}\right)^{[1,\ldots,j;1,\ldots,j]}\big)}{\det (z I_{n}-T_{n}\big)}
$$
for $n\to\infty$ and for fixed $j=1,\ldots,p$. These limiting ratios are hard to deal with. Therefore we use a
different approach. We will use a reflection of $T_n$ with respect to its main anti-diagonal, which turns the above
ratios into expressions of the form
$$ \frac{\det (z
I_{n-j}-T_{n-j})}{\det (z I_{n}-T_{n})} \equiv \frac{P_{n-j}(z)}{P_{n}(z)}.
$$
Note that the matrix in the numerator of the left hand side, is obtained from the matrix in the denominator by skipping
its \emph{last} $j$ rows and columns (rather than the \emph{first} $j$ rows and columns, thanks to the anti-diagonal
reflection.) The limiting ratios for $n\to\infty$ will then be obtained with the help of equation \eqref{ratio:asy:PQ}
(depending on the residue class of $n$ modulo $p$).

\smallskip
Now we work out the above ideas in detail. For a fixed $k\in\{0,\ldots,p-1\}$, let
\[
T^{(k)}:=T^{[1,\ldots,k;1,\ldots,k]},
\]
i.e., $T^{(k)}$ is the infinite matrix obtained by deleting the first $k$ rows and columns of $T$. Of course, we
understand that $T^{(0)}=T$. Note that $T^{(k)}$ is also a tridiagonal block Toeplitz matrix
\[
T^{(k)}=\begin{pmatrix} B_{0}^{(k)} & B_{-1}^{(k)}\\ B_{1}^{(k)} & B_{0}^{(k)} & B_{-1}^{(k)} \\
 & B_{1}^{(k)} & B_{0}^{(k)} & B_{-1}^{(k)} \\
 & & B_{1}^{(k)} & B_{0}^{(k)} & \ddots \\
 & & & \ddots & \ddots
\end{pmatrix},
\]
where the blocks of $T^{(k)}$ are also of size $p\times p$ and can be computed easily. Let $\Pi_{p}$ be the $p\times p$
permutation matrix that consists of $1$'s in the main anti-diagonal and $0$'s elsewhere, i.e.
\begin{equation}\label{def:permut}
\Pi_{p}(i,j)=\left\{\begin{array}{cl} 1, & \mbox{if}\,\, i+j=p+1,\\ 0, & \mbox{otherwise},
\end{array}
\right.
\end{equation}
where $\Pi_{p}(i,j)$ represents the entry in row $i$ and column $j$ of $\Pi_{p}$. We construct now a new block Toeplitz
matrix $\widetilde{T}^{(k)}$ as follows:
\[
\widetilde{T}^{(k)}=\begin{pmatrix} \widetilde{B}_{0}^{(k)} & \widetilde{B}_{-1}^{(k)}\\ \widetilde{B}_{1}^{(k)} &
\widetilde{B}_{0}^{(k)} & \widetilde{B}_{-1}^{(k)} \\
 & \widetilde{B}_{1}^{(k)} & \widetilde{B}_{0}^{(k)} & \widetilde{B}_{-1}^{(k)} \\
 & & \widetilde{B}_{1}^{(k)} & \widetilde{B}_{0}^{(k)} & \ddots \\
 & & & \ddots & \ddots
\end{pmatrix},
\]
where
\[
\widetilde{B}_{i}^{(k)}=\Pi_{p} (B_{i}^{(k)})^{T} \Pi_{p},\qquad i=-1,0,1,
\]
where the superscript ${}^T$ stands for the matrix transposition. So the block $\widetilde{B}_{i}^{(k)}$ is obtained by
reflecting $B_{i}^{(k)}$ with respect to its main anti-diagonal.

\begin{lemma}\label{propformula}
Fix $n=mp+k$, with $m\in\zet_{\geq 0}$ and $k\in\{0,1,\ldots,p-1\}$. Then the Chebyshev-Nikishin polynomials
$\{P_{n-j}\}_{j=0}^{n}$ are obtained by the formulas
\begin{equation}\label{detformPn}
P_{n-j}(z)=\det (z I_{n-j}-(\widetilde{T}^{(k)}_n)^{[1,\ldots,j;1,\ldots,j]}),\qquad 0\leq j\leq n,
\end{equation}
where $I_{n-j}$ is the identity matrix of size $n-j$ and $\widetilde{T}^{(k)}_n$ is the truncation of
$\widetilde{T}^{(k)}$ to the first $n$ columns and rows.
\end{lemma}
\begin{proof}
Formula \eqref{detformPn} follows directly from the relation $\widetilde{T}^{(k)}_n=\Pi_{n}\, T_{n}\, \Pi_{n}^{T}$,
where $\Pi_{n}$ is the $n\times n$ permutation matrix defined as in \eqref{def:permut}.
\end{proof}

We keep $n=mp+k$ fixed in what follows. If we set $A=\widetilde{T}^{(k)}$, then Lemma \ref{propformula} asserts that
the polynomials $\{P_{n-j}\}_{j=0}^{p}$ are of the form $D_{n}^{(j)}$, cf. \eqref{defDn}--\eqref{defnDnj}. So (see
Remark \ref{rmkHP}) we know that
\begin{equation}\label{HPapprox}
\Big(\frac{P_{n-1}}{P_{n}},\frac{P_{n-1}+c_{1,1} P_{n-2}}{P_{n}},\ldots,
\frac{P_{n-1}+c_{1,p-1}P_{n-2}+c_{2,p-1}P_{n-3}+\cdots+c_{p-1,p-1}P_{n-p}}{P_{n}}\Big)
\end{equation}
is a Hermite-Pad\'{e} approximant to the system of functions
\[
(g_{1}^{(k)},g_{1}^{(k)}+c_{1,1}g_{2}^{(k)},g_{1}^{(k)}+c_{1,2}g_{2}^{(k)}+c_{2,2}g_{3}^{(k)},
\ldots,g_{1}^{(k)}+c_{1,p-1}g_{2}^{(k)}+\cdots+c_{p-1,p-1}g_{p}^{(k)}),
\]
with respect to the multi-index \eqref{multi:index} and the functions $g_{i}^{(k)}$ are the resolvent functions
\eqref{resolvf} associated with the operator $\widetilde{T}^{(k)}$.

\section{Proof of Proposition~\ref{prop:QP:samelimit}}\label{section:ratioasy}

In this section we prove Proposition~\ref{prop:QP:samelimit} on the ratio asymptotics of the Chebyshev-Nikishin
polynomials $P_n$. In the next section we will use this to prove Theorems~\ref{theorem:main:ortho} and
\ref{theorem:main:Nikishin}.

Recall the notations in Section~\ref{subsection:aboutproof}. Since the block Toeplitz matrix $T$ in
\eqref{defH}--\eqref{blocks:Bminus1} has a lower Hessenberg shape, each set $\Gamma_k$ in \eqref{defGammas} is
non-empty and is a finite union of analytic arcs, see \cite[Prop.~1.1 and Lemma 2.1]{Del}. From \cite{Apt} or
\cite[Example~5.3]{Del} (or straightforward verification) we also know that the roots $z_k(x)$ in
\eqref{ordering:roots} have the asymptotics
\begin{equation}\label{growth:inf}
\begin{array}{ll}
z_{1}(x) =\frac{1}{x^{p}}+O\Big(\frac{1}{x^{p+1}}\Big), & x\rightarrow\infty,\\[1em] C_1|x|<|z_k(x)|<C_2|x|, &
x\rightarrow\infty,\quad 2\leq k\leq p+1,
\end{array}
\end{equation}
for certain constants $C_1,C_2>0$. Therefore $\Gamma_{1}$ is compact.\smallskip

We also need some results and notations from \cite{AptLopRoc}. Let $\mathcal{R}$ denote the compact Riemann surface
\[
\mathcal{R}=\overline{\bigcup_{k=0}^{p}\mathcal{R}_{k}}
\]
formed by gluing in the usual crosswise manner the consecutive sheets
\[
\mathcal{R}_{0}:=\overline{\mathbb{C}}\setminus\Delta_{1},\qquad \mathcal{R}_{k}=\overline{\mathbb{C}}\setminus
(\Delta_{k}\cup\Delta_{k+1}),\quad k=1,\ldots,p-1,\qquad \mathcal{R}_{p}:=\overline{\mathbb{C}}\setminus\Delta_{p}.
\]
Let $\psi^{(k)}$, $k=1,\ldots,p$, denote the meromorphic function on $\mathcal{R}$ whose divisor consists of a simple
pole at $\infty^{(0)}\in\mathcal{R}_{0}$ and a simple zero at $\infty^{(k)}\in\mathcal{R}_{k}$, with the normalization
\[
\psi^{(k)}(z)=z+O(1),\qquad z\rightarrow\infty^{(0)}.
\]
Here $\infty^{(k)}$ denotes the point at infinity on the sheet $\mathcal R_k$. Let $\psi_{j}^{(k)}$ denote the
restriction of $\psi^{(k)}$ to the sheet $\mathcal{R}_{j}$.  The main result of \cite{AptLopRoc} asserts that
\begin{equation}\label{ratio:asy:psi}
\lim_{m\rightarrow\infty}\frac{Q_{mp+k}(z)}{Q_{mp+k-1}(z)}=:F_{k}(z)=\psi_{0}^{(k)}(z),\qquad
z\in\mathbb{C}\setminus\Delta_{1},\quad k=1,\ldots,p,
\end{equation}
recall \eqref{ratioasymp}. From Liouville's Theorem we deduce that for any pair of indices $j,k$, the function
$\psi^{(j)}-\psi^{(k)}$ is constant on $\mathcal{R}$. In particular, $F_{j}-F_{k}$ is also constant.

\smallskip From the recurrence
\[
z Q_{mp+k}=Q_{mp+k+1}+a_{mp+k,mp+k}Q_{mp+k}+\cdots+a_{mp+k,mp+k-p} Q_{mp+k-p},
\]
dividing by $Q_{mp+k}$ and taking the limit as $m\to \infty$, we deduce that for each $k\in\{0,\ldots,p-1\}$,
\begin{equation}\label{eq:relFb}
z=F_{k+1}+b_{k,k}+b_{k,k-1}\frac{1}{F_{k}}+b_{k,k-2}\frac{1}{F_{k}\,F_{k-1}} +\cdots+b_{k,k-p}\frac{1}{F_{k}\cdots
F_{k-p+1}},
\end{equation}
with the obvious identifications $b_{k,k-i}=b_{p+k,p+k-i}$ if $i>k$ and $F_{k-i}=F_{p+k-i}$ if $i\geq k$
(cf.~\eqref{ratioasymp} and \eqref{def:limitreccoeff}). Multiplying both sides of \eqref{eq:relFb} by $F_{k}\cdots
F_{k-p+1}$, we obtain
\begin{equation}\label{eq:aux3}
z F_{k}\cdots F_{k-p+1}=F_{k+1}F_{k}\cdots F_{k-p+1}+b_{k,k}F_{k}\cdots F_{k-p+1}+\cdots
+b_{k,k-p+1}F_{k+1}+b_{k,k-p}.
\end{equation}
Except for the term $b_{k,k-p}$, all the other terms in \eqref{eq:aux3} contain a factor $F_{k+1}\equiv F_{k-p+1}$.
Writing all the functions $F_{j}$, $j\neq k+1$, as $F_{k+1}$ plus some constant, we arrive at an algebraic equation of
degree $p+1$ satisfied by $F_{k+1}(z)=\psi_{0}^{(k+1)}(z)$ of the form
\[
F_{k+1}(z)^{p+1}+C_{p}(z)\, F_{k+1}(z)^{p}+\cdots+C_{1}(z)\, F_{k+1}(z)+b_{k,k-p}=0.
\]
By analytic continuation, this equation is also satisfied by the other branches $\psi_{i}^{(k+1)}$, $i=1,\ldots,p$. It
follows that
\begin{equation}\label{equation:psi:k}
b_{k,k-p}\equiv (-1)^{p+1}\prod_{i=0}^{p}\psi_{i}^{(k+1)}(z),\qquad z\in\overline{\mathbb{C}}.
\end{equation}
In particular we have that $f_{p}\neq 0$, see \eqref{def:fp}. This allows us to write \eqref{Widom:Pn}.
\smallskip

Recall the strong asymptotics in \eqref{strongasy:Pn}. We remark that the function $\det F^{[p,k+1]}(z_{1}(x),x)$ in
\eqref{strongasy:Pn} has at most a finite number of zeros outside $\Gamma_{1}$. Indeed, the contrary would mean that
$\det F^{[p,k+1]}(z_{1}(x),x)$ is identically zero as a function of $x$, which by analytic continuation would imply
that \emph{each} of the coefficients in Widom's formula \eqref{Widom:Pn} for $P_{mp+k}(x)$ is identically zero. But
then $P_{mp+k}(x)\equiv 0$, which is clearly a contradiction; see also \cite[Lemma 5.7]{Del}.

From \eqref{strongasy:Pn} we deduce that
\begin{equation}\label{ratiolim1}
\lim_{n\rightarrow\infty}\frac{P_{n}(x)}{P_{n+p}(x)}=z_1(x)
\end{equation}
and
\begin{equation}\label{ratiolim2}
\lim_{m\rightarrow\infty}\frac{P_{mp+k}(x)}{P_{mp}(x)} = (-1)^{k}\,\frac{\det F^{[p,k+1]}(z_{1}(x),x)}{\det
F^{[p,1]}(z_{1}(x),x)},\qquad k=0,\ldots,p-1,
\end{equation}
uniformly on each compact subset of $\cee\setminus(\Gamma_1\cup\mathcal A)$ with $\mathcal A$ a finite set.\\

We want to relate the ratio asymptotics in \eqref{ratiolim2} and \eqref{ratio:asy:psi}. Recall that all the zeros of
all the polynomials $Q_n$ are contained in the interval $\Delta_1$, and that the zeros of $Q_n$ and $Q_{n+1}$ interlace
(see \cite[Theorem 2.1]{AptLopRoc}). We can then apply \cite[Section 4]{BDK} (taking into account that we do not know
yet that $\Gamma_1\subset\er$) and obtain for all $x\in\cee$ sufficiently large that
$$ \lim_{n\to\infty} \frac{Q_n(x)}{Q_{n+p}(x)}=z_1(x).
$$
Moreover, \cite[Remark 4.2]{BDK} shows that the vector
$$ \mathbf{Q}(x) :=  \lim_{m\to\infty}
\frac{1}{Q_{mp}(x)}\begin{pmatrix}Q_{mp}(x)\\ \vdots \\ Q_{mp+p-1}(x)\end{pmatrix}$$ satisfies the matrix-vector
product relation
\begin{equation*}
F(z_1(x),x)\mathbf{Q}(x)=\mathbf{0},
\end{equation*}
for all $x\in\cee$ sufficiently large. From the interlacing of zeros we know that all the limiting ratios
$\lim_{m\to\infty} Q_{mp+k}(x)/Q_{mp}(x)$ are analytic in $\cee\setminus\Delta_1$. Thus we get by analytic continuation
that
\begin{equation}\label{cramer0} \lim_{n\to\infty}
\frac{Q_n(x)}{Q_{n+p}(x)}=\til z_1(x),\qquad x\in\cee\setminus\Delta_1,
\end{equation}
and
\begin{equation}\label{cramer1}
F(\til z_1(x),x)\mathbf{Q}(x)=\mathbf{0},\qquad x\in\cee\setminus\Delta_1,
\end{equation}
where $\til z_1(x)$ is the (unique) root of the algebraic equation $f(z,x)=0$ that satisfies $\til z_1(x) \sim x^{-p}$
for $x\to\infty$ and that depends analytically on $x\in\cee\setminus\Delta_1$. Obviously $\til z_1(x)=z_1(x)$ for all
$x$ large enough. We will see further that $\til z_1(x)=z_1(x)$ holds for \emph{all} $x\in\cee\setminus\Delta_1$.

From the above discussion we see in particular that the root $z_1(x)$ can be analytically continued from a neighborhood
of infinity to $\cee\setminus\Delta_1$.

Applying Cramer's rule to \eqref{cramer1} gives us
\begin{align}
\lim_{m\rightarrow\infty}\frac{Q_{mp+k}(x)}{Q_{mp}(x)} & =(-1)^{k}\,\frac{\det F^{[p,k+1]}(\til z_{1}(x),x)}{\det
F^{[p,1]}(\til z_{1}(x),x)},\qquad x\in\cee\setminus\Delta_1, \label{ratiolim1:Q}
\end{align}
for all  $k=0,\ldots,p-1$. Comparing this with \eqref{ratiolim1}--\eqref{ratiolim2}, we see that the proof of
\eqref{ratio:asy:PQ} will be obtained once we know that $\til z_1(x)=z_1(x)$ for all $x\in\cee\setminus\Delta_1$.

Incidentally, we point out that an alternative proof of \eqref{cramer0} and \eqref{ratiolim1:Q} could be constructed
with the help of the generalized Poincar\'e theorem in \cite{MN}, rather than using \cite{BDK}.
\smallskip

From \eqref{ratio:asy:psi} and \eqref{cramer0} we see that
\begin{equation}\label{ztilde:product} \til z_1(x) = \frac{1}{\psi_0^{(1)}(x)\ldots
\psi_0^{(p)}(x)}= \frac{1}{F_{1}(x)\ldots F_{p}(x)},\qquad x\in\cee\setminus\Delta_1.
\end{equation}
The right hand side can be analytically extended to a meromorphic function on the Riemann surface $\err$, with
restriction to the sheet $\err_{k-1}$ given by
\begin{equation}\label{ztilde:product:bis} \til z_k(x) := \frac{1}{\psi_{k-1}^{(1)}(x)\ldots
\psi_{k-1}^{(p)}(x)},\qquad x\in\cee\setminus(\Delta_{k-1}\cup\Delta_{k}),
\end{equation}
for $k=1,\ldots,p+1$, where we understand that $\Delta_0=\emptyset=\Delta_{p+1}$. Thus $\til z_k(x)$ and $\til
z_{k+1}(x)$ are each others analytic continuation across $\Delta_k$. From the uniqueness of analytic continuation we
then get that for each $x\in\cee$, $\til z_1(x),\ldots,\til z_{p+1}(x)$ are a permutation of the roots
$z_1(x),\ldots,z_{p+1}(x)$ of the algebraic equation $f(z,x)=0$. If we would be able to prove that
\begin{equation}\label{ordering:roots:tilde}
|\til z_{1}(x)|\leq |\til z_{2}(x)|\leq \cdots\leq |\til z_{p+1}(x)|,\qquad x\in\mathbb{C},
\end{equation}
and \begin{equation}\label{ordering:roots:strict} |\til z_k(x)|<|\til z_{k+1}(x)|, \qquad
x\in\cee\setminus\Delta_k,\end{equation} then we could conclude from \eqref{ordering:roots} that $\til z_k(x)=z_k(x)$
for all $x\in\cee\setminus\left(\Delta_{k-1}\cup\Delta_k\right)$.
\\

The rest of the proof is devoted to proving \eqref{ordering:roots:tilde}--\eqref{ordering:roots:strict}. We will do
this by a rather intricate argument involving the total masses of certain equilibrium measures. An alternative proof
may be found in \cite[Propositions 2.1--2.2]{Apt}.

To start, we define a measure $s_k$ on $\Delta_k$ with density
\begin{equation}\label{rho:0} \ud s_k(x)=\frac{1}{2\pi i}\left( \frac{\til
z_{k,+}'(x)}{\til z_{k,+}(x)}- \frac{\til z_{k,-}'(x)}{\til z_{k,-}(x)} \right) \ud x,\qquad x\in\Delta_k,
\end{equation}
$k=1,\ldots,p$, where the prime denotes the derivative with respect to $x$, and where the $+$ and $-$ subscripts stand
for the boundary values obtained from the upper or lower half of the complex plane respectively. Note that the density
\eqref{rho:0} is well-defined in the interior of $\Delta_k$, and blows up at worst like an inverse square root at the
endpoints of $\Delta_k$ (this follows from the representation of $\til z_{k}$ in local coordinates near the endpoints).
We claim that $s_k$ is a real-valued (possibly signed) measure on $\Delta_k$ with total mass
\begin{equation}\label{rho:1} s_k(\Delta_k) := \int_{\Delta_k} \ud s_k(x)
= p+1-k,\qquad k=1,\ldots,p.
\end{equation}

To prove Equation~\eqref{rho:1}, we first derive the following relation for the Cauchy transforms:
\begin{equation}\label{rho:2} \int_{\Delta_{k-1}} \frac{1}{x-t} \ud s_{k-1}(t)
-\int_{\Delta_k} \frac{1}{x-t} \ud s_k(t) = \frac{\til z_k'(x)}{\til z_k(x)},\qquad
x\in\cee\setminus(\Delta_{k-1}\cup\Delta_k),
\end{equation}
for $k=1,\ldots,p$, where we understand that $\Delta_0=\emptyset$ and $s_0=0$. Indeed, equation \eqref{rho:2} follows
easily by contour deformation and by using that $\til z_{k-1,\pm}(x)=\til z_{k,\mp}(x)$ for $x\in\Delta_{k-1}$.

From
\begin{align*}
\til z_{1}(x) \sim x^{-p},& \qquad x\rightarrow\infty,\\ \til z_k(x)\sim C_k x & ,\qquad x\rightarrow\infty,\quad 2\leq
k\leq p+1,
\end{align*}
we get for the logarithmic derivatives that
$$ \frac{\til z_k'(x)}{\til
z_k(x)} \sim \left\{\begin{array}{ll} -p/x,& \quad k=1,\\ 1/x,& \quad k=2,\ldots,p,\end{array}\right.
$$
for $x\to\infty$. Thus by equating the $1/x$ terms in the asymptotics for $x\to\infty$ in both sides of \eqref{rho:2}
we obtain
$$ s_{k-1}(\Delta_{k-1})-s_k(\Delta_k) = \left\{\begin{array}{ll} -p,& \quad k=1,\\ 1,& \quad
k=2,\ldots,p.\end{array} \right.
$$
The claim \eqref{rho:1} follows from this by means of an upward induction on $k=1,\ldots,p$.
\smallskip

The symmetry under complex conjugation shows that $\til z_k(\bar x)=\overline{\til z_k(x)}$ where the bar denotes the
complex conjugation. We then obtain from \eqref{rho:0}--\eqref{rho:1} that
\begin{equation}\label{totalmass:1} \frac{1}{\pi}\int_{\Delta_k}
\Im\left(\frac{\til z_{k,+}'(x)}{\til z_{k,+}(x)}\right) \ud x = p+1-k,
\end{equation}
with $\Im$ denoting the imaginary part of a complex number.
\\

On the other hand, consider the (positive) measure $\sigma_k$ on $\Gamma_k$ with density \cite{Del}
\begin{equation}\label{sigma:density}
\ud\sigma_k(x)=\frac{1}{2\pi i}\sum_{j=1}^{k}\left( \frac{z_{j,+}'(x)}{z_{j,+}(x)}- \frac{z_{j,-}'(x)}{z_{j,-}(x)}
\right)\ud x,\qquad x\in\Gamma_{k}.
\end{equation}
This measure is well-defined on each open analytic arc of $\Gamma_k$, recalling that $\Gamma_k$ is a finite union of
analytic arcs. Note that we are using the roots $z_j$ and not $\til z_j$. The measure $\sigma_{k}$, restricted to the
real line, takes the form
\begin{equation}\label{rho:3}
\ud\sigma_{k}(x)=\frac{1}{\pi}\sum_{j=1}^{k}\Im\left( \frac{z_{j,+}'(x)}{z_{j,+}(x)}\right)\ud x,\qquad
x\in\Gamma_k\cap\mathbb{R}.
\end{equation}

Now, let us take a fixed open interval $J\subset\er$ that does not contain any intersection points or endpoints of the
analytic arcs constituting $\Gamma_k$, for every $k$. We also ask $J$ not to contain isolated intersection points of
the sets $\Gamma_{k}$ with the real axis. Thus there exists an open connected set $U\subset\cee$ such that $U\cap
\er=J$ and moreover $U\cap\Gamma_k$ is either empty or equal to $J$, for any $k=1,\ldots,p$. The boundary values
$z_{k,+}(x)$ for $x\in J$ are then uniquely defined and they vary analytically with $x$.

On the interval $J$, there exist indices $ 1\leq m_1<m_2<\ldots <m_L\leq p$ such that
\begin{multline}\label{clusters} |z_{1,+}(x)|=\ldots =
|z_{m_1,+}(x)|<|z_{m_1+1,+}(x)|=\ldots = |z_{m_2,+}(x)|<\ldots
\\ <|z_{m_L+1,+}(x)|=\ldots = |z_{p+1,+}(x)|,
\end{multline}
for all $x\in J$. We define $m_0:=0$ and $m_{L+1}:=p+1$.

We will see later that $m_{k+1}-m_k\in\{1,2\}$ for all $k$, i.e., each ``cluster" $|z_{m_k+1,+}(x)|=\ldots =
|z_{m_{k+1},+}(x)|$ in \eqref{clusters} can only have length 1 or 2.

From the ordering \eqref{ordering:roots} we find for each of the clusters $|z_{m_k+1,+}(x)|=\ldots =
|z_{m_{k+1},+}(x)|$ in \eqref{clusters} that
$$ \left[\frac{\ud}{\ud\epsilon}\left(\Re\log z_{m_k+1}(x+i\epsilon) \right)\right]_{\epsilon=0+}\leq
\ldots \leq \left[\frac{\ud}{\ud\epsilon}\left(\Re\log z_{m_{k+1}}(x+i\epsilon)\right)\right]_{\epsilon=0+},\qquad x\in
J,
$$
for any $k=0,\ldots,L$. Using the Cauchy-Riemann equations this implies
\begin{equation}\label{pairing:1} \Im\left( \frac{z_{m_k+1,+}'(x)}{z_{m_k+1,+}(x)} \right)\geq \ldots \geq
\Im\left( \frac{z_{m_{k+1},+}'(x)}{z_{m_{k+1},+}(x)} \right),\qquad x\in J,
\end{equation}
$k=0,\ldots,L$.

By the symmetry with respect to complex conjugation we have a pairing of the numbers in \eqref{pairing:1}: they appear
in positive-negative pairs in the sense that \begin{equation}\label{pairing:2} \Im\left(
\frac{z_{m_k+j,+}'(x)}{z_{m_k+j,+}(x)} \right) = -\Im\left( \frac{z_{m_{k+1}+1-j,+}'(x)}{z_{m_{k+1}+1-j,+}(x)}
\right),\qquad j=1,\ldots,m_{k+1}-m_k,
\end{equation}
for any $x\in J$ and $k=0,\ldots,L$. Note that the numbers in \eqref{pairing:1} can also be identically zero, in the
case of a real-valued root. Clearly there is at least one real root if $m_{k+1}-m_k$ is odd.

Let $\til m_k:=\left\lfloor \frac{m_k+m_{k+1}}{2}\right\rfloor$. From the above considerations we find for the density
of the measure $\sigma_{\til m_k}$ in \eqref{rho:3} that
\begin{eqnarray*} \frac{\ud\sigma_{\til m_k}(x)}{\ud x} &=& \frac{1}{\pi}\sum_{j=m_k+1}^{\til
m_k} \Im\left( \frac{z_{j,+}'(x)}{z_{j,+}(x)} \right)\\ &=& \frac{1}{2\pi}\sum_{j=m_k+1}^{m_{k+1}} \left|\Im\left(
\frac{z_{j,+}'(x)}{z_{j,+}(x)} \right)\right|,\qquad x\in J,
\end{eqnarray*}
for any $k=0,\ldots,L$, where the first equality follows from the cancelations arising from \eqref{pairing:2}, and the
second equality follows from \eqref{pairing:1}--\eqref{pairing:2}. Hence the sum of the total masses of all the
measures $\sigma_k$ over the interval $J$ can be bounded from below by
\begin{eqnarray}\label{somesigmas} \sigma_{1}(J)+\ldots+\sigma_{p}(J) &\geq & \sigma_{\til
m_0}(J)+\ldots+\sigma_{\til m_L}(J)
\\ \nonumber &=& \frac{1}{2\pi}\int_{J}\sum_{j=1}^{p+1} \left|\Im\left(
\frac{z_{j,+}'(x)}{z_{j,+}(x)} \right)\right|\ud x.
\end{eqnarray}
Putting $\Delta:=\bigcup_{k=1}^p \Delta_k$ we then obtain
\begin{equation}\label{pairing:3}
\sigma_{1}(\Delta)+\ldots+\sigma_{p}(\Delta) \geq \frac{1}{2\pi}\int_{\Delta}\sum_{j=1}^{p+1} \left|\Im\left(
\frac{z_{j,+}'(x)}{z_{j,+}(x)} \right)\right|\ud x.
\end{equation}
On the other hand, the (positive) measure $\sigma_k$ satisfies \cite[Cor.~4.2]{Del} (or see the proof of
\eqref{rho:1})
\begin{equation}\label{Gammak:er} \sigma_k(\Delta)\leq \sigma_{k}(\cee) =
\sigma_{k}(\Gamma_k) = p+1-k,
\end{equation}
for $k=1,\ldots,p$. Using this in \eqref{pairing:3} we get
\begin{eqnarray*}
\sum_{k=1}^{p} (p+1-k) & \geq & \sum_{k=1}^{p}\sigma_k(\Delta)\\ &\geq & \frac{1}{2\pi}\int_{\Delta}\sum_{k=1}^{p+1}
\left|\Im\left( \frac{z_{k,+}'(x)}{z_{k,+}(x)} \right)\right|\ud x\\ &= & \frac{1}{2\pi}\int_{\Delta}\sum_{k=1}^{p+1}
\left|\Im\left( \frac{\til z_{k,+}'(x)}{\til z_{k,+}(x)} \right)\right|\ud x\\ &= & \frac{1}{\pi}\sum_{k=1}^{p}
\int_{\Delta_k}\left|\Im\left( \frac{\til z_{k,+}'(x)}{\til z_{k,+}(x)} \right)\right|\ud x\\ &\geq& \sum_{k=1}^{p}
(p+1-k),
\end{eqnarray*}
where the second relation is \eqref{pairing:3}, and the third one follows since the roots $\til z_k$ form a permutation
of the roots $z_k$. The fourth relation uses that on $\Delta_k$ there are precisely two non-real roots $\til z_k(x)$
and $\til z_{k+1}(x)$, which are complex conjugated. Finally the last inequality is \eqref{totalmass:1}.

From the above chain of inequalities we find:

\begin{lemma}\label{lemma:chain}
\begin{enumerate}
\item[$(a)$] We have
\begin{equation}\label{Gammak:Delta}\Gamma_k\subset\Delta=\bigcup_{j=1}^{p}\Delta_j,\qquad
k=1,\ldots,p.\end{equation} \item[$(b)$] Clusters of length $\geq 3$ in \eqref{clusters} cannot occur unless they
have all entries in
    \eqref{pairing:1} equal to zero.
\end{enumerate}
\end{lemma}

\begin{proof}[Proof of Lemma~\ref{lemma:chain}]
(a) By comparing the two outermost terms in the above chain of inequalities, we see that each of the inequalities
$\geq$ is actually an equality. Since the first inequality in the chain comes from \eqref{Gammak:er}, this implies in
particular that $\sigma_k(\Delta)=\sigma_k(\cee)=\sigma_k(\Gamma_k)=p+1-k$ for all $k$. This implies
\eqref{Gammak:Delta}, since each analytic arc of $\Gamma_k$ has a strictly positive mass under $\sigma_k$.

Let us prove the statement made in the last sentence. We note that the product $z_1(x)\ldots z_k(x)$ is analytic in
$\cee\setminus\Gamma_k$ \cite[Prop.~3.5]{DK} and moreover, by splitting the analytic arcs of $\Gamma_k$ into finitely
many sub-arcs if necessary, the analytic continuation of $z_1(x)\ldots z_k(x)$ across each such sub-arc is of the form
$z_{j_1}(x)\ldots z_{j_k}(x)$ with $\{j_1,\ldots,j_k\}\neq\{1,\ldots,k\}$ \cite[Prop.~3.4]{DK}. This then implies that
$\log |z_1(x)\ldots z_k(x)|$ is nowhere harmonic on $\Gamma_k$. Since $-\log|z_1(x)\ldots z_k(x)|+C$ (for some constant
$C$) is the logarithmic potential of the measure $\sigma_k$ in \eqref{sigma:density} \cite[Prop.~5.10]{Del}, we then
conclude from the non-harmonicity that $\sigma_k$ is indeed nontrivial on each sub-arc of $\Gamma_k$. This proves the
claim.

(b) From the strictness of the inequalities $\geq$ in the above chain of inequalities, we also get that the inequality
$\geq$ in \eqref{somesigmas} is an equality. This is easily seen to imply Part (b).
\end{proof}

Now we are in a position to prove \eqref{GammaDelta}. First we will prove that $\Gamma_1=\Delta_1$. From
\eqref{Gammak:Delta} we already know that $\Gamma_1\subset\Delta$. Then the equality $\til z_1(x)=z_1(x)$, which is
known to hold for all $x$ sufficiently large, remains valid for all $x\in\cee\setminus\Delta$. Using this, we find
$$ z_{1,+}(x)=\til z_{1,+}(x) = \til z_{1,-}(x) = z_{1,-}(x),\qquad
x\in\Delta\setminus\Delta_1,
$$
where we used that $\til z_1(x)$ is analytic away from $\Delta_{1}$. Hence $z_1(x)$ is analytic on
$\Delta\setminus\Delta_1$ and then \cite[Prop.~3.4]{DK} implies that $(\Delta\setminus \Delta_{1})\cap
\Gamma_1=\emptyset$ and so $\Gamma_1\subset\Delta_1$. The reverse inclusion $\Delta_1\subset\Gamma_1$ is immediate
since $z_{1,+}(x)=\til z_{1,+}(x)=\overline{\til z_{1,-}(x)}=\overline{z_{1,-}(x)}$ for $x\in\Delta_1$, yielding two
different roots with the same modulus. This proves that $\Gamma_1=\Delta_1$.

From the non-triviality of $\sigma_1$ on $\Gamma_1=\Delta_1$ (see the proof of Lemma~\ref{lemma:chain}(a)), we have
that $\Im\left( \frac{z_{1,+}'(x)}{z_{1,+}(x)}\right)\neq 0$ for almost all $x\in\Delta_1$. Thus the cluster $\Im\left(
\frac{z_{1,+}'(x)}{z_{1,+}(x)} \right)\geq \ldots \geq \Im\left( \frac{z_{m_{1},+}'(x)}{z_{m_{1},+}(x)} \right)$ on the
interval $\Delta_1$ has length $m_1\leq 2$, in view of Lemma~\ref{lemma:chain}(b). On the other hand, we also have that
$m_1\geq 2$, since $\til z_{1,\pm}(x)=\til z_{2,\mp}(x)$ are complex conjugate for $x\in\Delta_1$. Thus $m_1=2$. This
also shows that $z_2(x)=\til z_2(x)$ for all $x$ near $\Delta_1$ and therefore for all $x\in\cee\setminus\Delta$, since
$\Gamma_2\subset\Delta$ (again from \eqref{Gammak:Delta}).

Now we prove that $\Gamma_2=\Delta_2$. In view of the last paragraph we have
$$ z_{1,+}(x)z_{2,+}(x)=\til z_{1,+}(x)\til z_{2,+}(x) = \til z_{1,-}(x)\til z_{2,-}(x) =
z_{1,-}(x)z_{2,-}(x),\qquad x\in\Delta\setminus\Delta_2,
$$
where we used that the product $\til z_1\til z_2$ is analytic away from $\Delta_2$. Thus $z_1z_2$ is analytic on
$\Delta\setminus \Delta_{2}$ and so \cite[Prop.~3.4]{DK} implies again that $(\Delta\setminus \Delta_{2})\cap
\Gamma_2=\emptyset$ and therefore $\Gamma_2\subset\Delta_2$. The reverse inclusion $\Delta_2\subset\Gamma_2$ follows
again since $z_{2,+}(x)=\til z_{2,+}(x)=\overline{\til z_{2,-}(x)}=\overline{z_{2,-}(x)}$ for $x\in\Delta_2$, yielding
two different roots with the same modulus. This proves that $\Gamma_2=\Delta_2$.

This reasoning can be extended to show that $\Gamma_k=\Delta_k$ for each $k=3,\ldots,p$. In particular, we have that
$\til z_k(x)=z_k(x)$ for $x\in\cee\setminus(\Delta_{k-1}\cup\Delta_k)$. This concludes the proof of
Prop.~\ref{prop:QP:samelimit}. $\bol$

\section{Proof of Theorem~\ref{theorem:main:ortho}}
\label{section:proof:main:ortho}

Using \eqref{ratio:asy:PQ} and a telescoping product we deduce that
\[
\lim_{m\rightarrow\infty}\frac{P_{mp+k-j}(x)}{P_{mp+k}(x)}  =\frac{1}{F_{k}(x) F_{k-1}(x) \cdots F_{k-j+1}(x)},\qquad
j=1,\ldots,p,\quad k=0,\ldots,p-1,
\]
where each $F_j$ is understood as $F_{j\,\textrm{mod}\, p}$. It follows that for $n=mp+k$, $0\leq k\leq p-1$, the
vector \eqref{HPapprox} is a Hermite-Pad\'{e} approximant to the system of functions
\[
\mathbf{G}^{(k)}=(G_{1}^{(k)},G_{2}^{(k)},\ldots,G_{p}^{(k)}),
\]
where
\begin{equation}\label{def:Gjk}
G_{j}^{(k)}(x)=\frac{1}{F_{k}(x)}+c_{1,j-1} \frac{1}{F_{k}(x) F_{k-1}(x)}+\cdots+c_{j-1,j-1}\frac{1}{F_{k}(x)
F_{k-1}(x)\cdots F_{k-j+1}(x)}.
\end{equation}
In particular, we obtain the following orthogonality conditions
\begin{equation}\label{eq:ortPn}
\int_{\gamma} P_{n}(x)\, x^{l} G_{j}^{(k)}(x)\ud x=0,\qquad l=0,\ldots,n_{j}-1,\quad j=1,\ldots,p,
\end{equation}
where $\gamma$ is an arbitrary closed Jordan curve surrounding $\Delta_{1}$. From \eqref{eq:ortPn} we see that the
polynomials $P_n$ satisfy certain multiple orthogonality relations, but with the weight functions $G_{j}^{(k)}(x)$
depending on the residue class $k$ of $n$ modulo $p$. To get around this issue, we will make a clever choice of the
constants $c_{i,j-1}$ in \eqref{def:Gjk}.

\begin{proposition}\label{prop:constants:cij} Let $j$ and $k$ be fixed.
The constants $c_{i,j-1}$ in \eqref{def:Gjk} can be chosen so that
\begin{equation}\label{constant:cij} G_{j}^{(k)}(x) = \frac{1}{F_{k-j+1}(x)},
\end{equation}
where again $F_{k-j+1}$ is understood as $F_{(k-j+1)\,\textrm{mod}\, p}$.
\end{proposition}

\begin{proof} The function $F_k(x)$ in \eqref{ratioasymp} has
a simple pole at $x=\infty$, with
\begin{equation}\label{def:Kj}
K_{k} := \lim_{x\to \infty} F_k(x) -x=-b_{k-1,k-1},
\end{equation}
by virtue of \eqref{eq:relFb} (with $z$ replaced by $x$). As already mentioned, Liouville's theorem implies that
$\psi^{(j)}- \psi^{(k)}$ is constant on $\mathcal{R}$, for any $j,k\in\{1,\ldots,p\}$. Combined with \eqref{def:Kj} and
\eqref{ratio:asy:psi} this yields
\begin{equation}\label{constants:Ki} F_j(x) -  F_k(x) =
\psi_0^{(j)}(x)- \psi_0^{(k)}(x) \equiv K_{j} -K_{k},\quad x \in \overline{\mathbb{C}} \setminus
\Delta_1.\end{equation}

Now let $j$ and $k$ be fixed. We will evaluate the sum \eqref{def:Gjk} from right to left. Formally, we do this by
defining functions $y_l(x)$, $l=0,\ldots,j$ with $y_0(x)=0$ and then recursively by
\begin{equation}\label{yl:def} y_l(x) = y_{l-1}(x)+c_{j-l,j-1}
\frac{1}{F_{k}(x) F_{k-1}(x)\cdots F_{k-j+l}(x)},
\end{equation} for $l=1,\ldots,j$, with $c_{0,j-1}:=1$. So $y_l(x)$ is the sum of the last $l$ terms in
\eqref{def:Gjk}, and in particular $y_{j}(x)=G_j^{(k)}(x)$.

We will prove by induction on $l$ that the constants $c_{j-l,j-1}$ can be chosen so that
\begin{equation}\label{constants:cij:aux} y_l(x)=
c_{j-l,j-1}\frac{1}{F_{k}(x) F_{k-1}(x)\cdots F_{k-j+l+1}(x)F_{k-j+1}(x)},
\end{equation}
for any $l=1,\ldots,j$. By taking $l=j$ this proves the proposition.

It remains to prove \eqref{constants:cij:aux}. For $l=1$ this is trivial. Now assume by induction that
\eqref{constants:cij:aux} holds with $l$ replaced by $l-1\in\{1,\ldots,j-1\}$. From the definition \eqref{yl:def} we
then get
$$ y_{l}(x)=
c_{j-l+1,j-1}\frac{1}{F_{k}(x) F_{k-1}(x)\cdots F_{k-j+l}(x)F_{k-j+1}(x)}+ c_{j-l,j-1}\frac{1}{F_{k}(x)
F_{k-1}(x)\cdots F_{k-j+l}(x)}.
$$
Putting on a common denominator yields
$$ y_{l}(x)=\frac{c_{j-l+1,j-1}+c_{j-l,j-1}F_{k-j+1}(x)}{F_{k}(x) F_{k-1}(x)\cdots F_{k-j+l}(x)F_{k-j+1}(x)}.
$$
Now by taking $c_{j-l+1,j-1}:=c_{j-l,j-1}(K_{k-j+l}-K_{k-j+1})$ and using \eqref{constants:Ki}, this becomes
$$ y_{l}(x)=c_{j-l,j-1}\frac{F_{k-j+l}(x)}{F_{k}(x) F_{k-1}(x)\cdots F_{k-j+l}(x)F_{k-j+1}(x)},
$$
which shows that \eqref{constants:cij:aux} holds for $l$. This proves the induction step.
\end{proof}

From \eqref{eq:ortPn} and \eqref{constant:cij}, we see that the Chebyshev-Nikishin polynomials $P_n(x)$ satisfy the
orthogonality conditions
\begin{equation}\label{eq:ortPn:bis}
\int_{\gamma} P_{n}(x)\, x^{l} \frac{1}{F_{j}(x)}\ud x=0,\qquad l=0,\ldots,n_{j}-1,\quad j=1,\ldots,p,
\end{equation}
with again $\gamma$ an arbitrary closed Jordan curve surrounding $\Delta_{1}$. Indeed, if $n=mp+k$ with
$k\in\{0,1,\ldots,p-1\}$, then \eqref{eq:ortPn:bis} can be checked  for each fixed residue class $k$ individually,
using the above considerations and the fact that the orthogonality measures with the same number of orthogonality
constraints can be freely interchanged.

By shrinking the contour $\gamma$ in \eqref{eq:ortPn:bis} to the interval $\Delta_1$ and using the Stieltjes-Perron
inversion formula, we get (see \eqref{cauchy1})
\begin{equation}\label{eq:ortPn:tris}
\int_{\Delta_1} P_{n}(x)\, x^{l} \ud\nu_j(x)=0,\qquad l=0,\ldots,n_{j}-1,\quad j=1,\ldots,p.
\end{equation}
This proves Theorem~\ref{theorem:main:ortho}. $\bol$

\section{Proof of Theorem~\ref{theorem:main:Nikishin}}
\label{section:proof:main:Nik}

The proof uses some ideas from Aptekarev-Kaliaguine-Saff \cite{AKS}, see also \cite{DelLop}.

Write $(\nu_1^{(1)},\ldots,\nu_p^{(1)}):=(\nu_1,\ldots,\nu_p)$. These measures will form the first layer of the
Nikishin hierarchy. Denote the Cauchy transform of $\nu_j^{(1)}$ by
\[ F_{j}^{(1)}(x):=1/F_{j}(x), \qquad j=1,\ldots,p,
\]
recall \eqref{cauchy1}. By the Stieltjes-Perron inversion principle, the density of $\nu_j^{(1)}$ is given by
\[
\frac{\ud \nu_j^{(1)}(x)}{\ud x} = \frac{1}{2\pi i}\left(F_{j,-}^{(1)}(x)-F_{j,+}^{(1)}(x)\right),\qquad x \in
\Delta_1,
\]
with again $\pm$ denoting the boundary values from the upper and lower half-plane respectively. With the help of
\eqref{ratio:asy:psi} this becomes
\[ \frac{\ud \nu_j^{(1)}(x)}{\ud x} = \frac{-1}{2\pi
i}\frac{\psi_{0,-}^{(j)}(x)-\psi_{0,+}^{(j)}(x)}{\psi_{0,+}^{(j)}(x)\psi_{0,-}^{(j)}(x)} = \frac{-1}{2\pi
i}\frac{\psi_{0,-}^{(j)}(x)-\psi_{0,+}^{(j)}(x)}{\psi_0^{(j)}(x)\psi_1^{(j)}(x)} = \frac{-1}{2\pi
i}\frac{\psi_{0,-}^{(1)}(x)-\psi_{0,+}^{(1)}(x)}{\psi_0^{(j)}(x)\psi_1^{(j)}(x)},
\]
where the last equality uses that, according to \eqref{constants:Ki},
\begin{equation} \label{relfund}
\psi^{(j)}(x) = \psi^{(k)}(x)+ K_j-K_k,\qquad 1\leq j,k\leq p,
\end{equation}
which implies
\[ \psi_{0,\pm}^{(j)}(x) = \psi_{0,\pm}^{(1)}(x)+ K_j-K_1.
\]
(Notice that $\psi_0^{(j)}\psi_1^{(j)} \in \mathcal{H}(\mathbb{C}\setminus \Delta_2)$, i.e., it is holomorphic in
$\mathbb{C}\setminus \Delta_2$.)

Let us prove that the ratio of densities of $\nu_j^{(1)}$  and $\nu_1^{(1)}$,
$$ F^{(2)}_{j}(x):=\frac{F_{j,-}^{(1)}(x)-F_{j,+}^{(1)}(x)}{F_{1,-}^{(1)}(x)-F_{1,+}^{(1)}(x)},\qquad x\in\Delta_1,
\qquad j=2,\ldots,p,
$$
can be analytically extended to $\cee\setminus\Delta_2$ and  can be written as a Cauchy transform
$$ F^{(2)}_{j}(x) = \int_{\Delta_2}\frac{1}{x-t}\ud\nu^{(2)}_j(t),\qquad
x\in\cee\setminus\Delta_2, \qquad j=2,\ldots,p,
$$
for some   measure $\nu^{(2)}_j$ with constant sign supported on $\Delta_2$. The measures
$(\nu^{(2)}_2,\ldots,\nu^{(2)}_p)$ form the second layer of the Nikishin hierarchy.

In fact
\[ F^{(2)}_{j}(x) = \frac{\psi_{0}^{(1)}(x)\psi_{1}^{(1)}(x)}{\psi_{0}^{(j)}(x)\psi_{1}^{(j)}(x)}, \qquad x \in
\Delta_1, \qquad j=2,\ldots,p,
\]
and thus
\begin{equation} \label{eq:F2} F^{(2)}_{j}  = \frac{\psi_{0}^{(1)} \psi_{1}^{(1)} }{\psi_{0}^{(j)} \psi_{1}^{(j)} }
\in
\mathcal{H}(\mathbb{C} \setminus \Delta_2), \qquad j=2,\ldots,p.
\end{equation}
Moreover, at infinity the numerator takes a finite value whereas the denominator has a simple pole, therefore
\[ F^{(2)}_{j}(x) = O\left(\frac{1}{x}\right), \qquad x \to \infty,
\]
having a simple zero at $\infty$ and no other zeros in $\mathbb{C} \setminus \Delta_2$.

Let $\gamma$ be a positively oriented closed Jordan curve surrounding $\Delta_2$. By Cauchy's integral formula we have
that
\[ F^{(2)}_{j}(x) = \frac{1}{2\pi i} \int_{\gamma} \frac{F^{(2)}_{j}(t)}{x-t} \ud t,
\]
for all $x$ exterior to $\gamma$. Shrinking $\gamma$ to $\Delta_2$  we get that
\[ F^{(2)}_{j}(x) = \int_{\Delta_2} \frac{\ud \nu_j^{(2)}(t)}{x-t},
\]
where
\[ \ud \nu_j^{(2)}(t) = \frac{1}{2\pi i}\left(F_{j,-}^{(2)}(t)-F_{j,+}^{(2)}(t)\right) \ud t,\qquad t \in \Delta_2.
\]
Since $\psi_0^{(k)} \in \mathcal{H}(\mathbb{C}\setminus \Delta_1), k=1,\ldots,p$, and $\Delta_1 \cap \Delta_2 =
\emptyset$, the density of $\nu^{(2)}_j$ takes the form
\[ \frac{\ud \nu^{(2)}_j(x)}{\ud x} = \frac{1}{2\pi i}\left(F_{j,-}^{(2)}(x)-F_{j,+}^{(2)}(x)\right) = \frac{1}{2\pi
i} \frac{\psi_{0}^{(1)}(x)}{\psi_{0}^{(j)}(x)} \left(\frac{\psi_{1,-}^{(1)}(x)}{\psi_{1,-}^{(j)}(x)} -
\frac{\psi_{1,+}^{(1)}(x)}{\psi_{1,+}^{(j)}(x)}\right), \quad x \in \Delta_2.
\]
On the other hand, using \eqref{relfund}, the expression in the last parentheses reduces to
\[ \frac{\psi_{1,-}^{(1)}(x)\psi_{1,+}^{(j)}(x)-\psi_{1,+}^{(1)}(x)\psi_{1,-}^{(j)}(x)}
{\psi_{1,-}^{(j)}(x)\psi_{1,+}^{(j)}(x)} = \frac{(K_j-K_1)(\psi_{1,-}^{(1)}(x)-\psi_{1,+}^{(1)}(x))
}{\psi_{1}^{(j)}(x)\psi_{2 }^{(j)}(x)}
\]
since $\psi_{1,-}^{(j)}(x)\psi_{1,+}^{(j)}(x) = \psi_{2,+}^{(j)}(x)\psi_{1,+}^{(j)}(x) =
\psi_{1,-}^{(j)}(x)\psi_{2,-}^{(j)}(x)$ can be extended analytically to a neighborhood of $\Delta_2$. Consequently,
\[ \frac{\ud \nu^{(2)}_j(x)}{\ud x} = \frac{K_j-K_1}{2\pi
i}\frac{\psi_{0}^{(1)}(x)(\psi_{1,-}^{(1)}(x)-\psi_{1,+}^{(1)}(x))}{\psi_{0}^{(j)}(x)\psi_{1
}^{(j)}(x)\psi_{2}^{(j)}(x)}, \qquad x \in \Delta_2,\qquad j=2,\ldots,p,
\]
where $\psi_{0}^{(j)}\psi_{1 }^{(j)}\psi_{2}^{(j)} \in \mathcal{H}(\mathbb{C} \setminus \Delta_3)$ and
$\psi_{0}^{(1)}\in \mathcal{H}(\mathbb{C} \setminus \Delta_1)$. Moreover, from the symmetry of these functions we
deduce that $\psi_{0}^{(1)}/\psi_{0}^{(j)}\psi_{1 }^{(j)}\psi_{2}^{(j)}$ takes real values with a constant sign on
$\Delta_2$. On the other hand,
\[ \psi_{1,-}^{(1)}(x)-\psi_{1,+}^{(1)}(x) = \psi_{1,-}^{(1)}(x)-\overline{\psi_{1,-}^{(1)}(x)} = 2i\, \Im
\{\psi_{1,-}^{(1)}(x)\}, \qquad x \in\Delta_2.\] Should $\psi_{1,-}^{(1)}(x)=0$ for some $x$ in the interior of
$\Delta_2$, we would have that at that point $\psi_{1,-}^{(1)}(x) = \psi_{1,+}^{(1)}(x)$. This is clearly impossible
because $\psi^{(1)}$ is one to one on $\mathcal{R}$ whereas $x_+$ and $x_-$ are distinct points on this Riemann
surface. Consequently for each $j=2,\ldots,p$, the measure $\nu_j^{(2)}$ has a constant sign on $\Delta_2$.

In general, fix $l\in\{2,\ldots,p\}$ and assume that we have defined measures
$(\nu^{(l-1)}_{l-1},\ldots,\nu^{(l-1)}_p)$ supported on $\Delta_{l-1}$, forming the $(l-1)$ layer of the Nikishin
hierarchy. Denote the Cauchy transform of $\nu^{(l-1)}_j$ by $F_{j}^{(l-1)}, j=l-1,\ldots,p$. Assume that we have also
shown that
\[ F_j^{(l-1)}(x) =
\frac{m_{l-2,j}}{m_{l-2,l-2}}\frac{\psi^{(l-2)}_0(x)\ldots \psi^{(l-2)}_{l-2}(x)}{\psi^{(j)}_0(x)\ldots
\psi^{(j)}_{l-2}(x)} \in \mathcal{H}(\mathbb{C} \setminus \Delta_{l-1}), \qquad (\psi_0^{(0)} \equiv 1),
\]
and
\[ \frac{\ud \nu^{(l-1)}_j(x)}{\ud x} = \frac{1}{2\pi
i}\left(F_{j,-}^{(l-1)}(x)-F_{j,+}^{(l-1)}(x)\right) =\]
\[ \frac{m_{l-1,j}}{2\pi i}
\frac{\psi^{(l-2)}_0(x)\ldots \psi^{(l-2)}_{l-3}(x)}{\psi^{(j)}_0(x)\ldots
\psi^{(j)}_{l-1}(x)}\left(\psi^{(1)}_{l-2,-}(x)-\psi^{(1)}_{l-2,+}(x)\right),
\]
for certain constants $m_{l-2,j}, m_{l-1,j} \in \mathbb{R} \setminus \{0\}$. (In fact, $m_{0,j}=1$, $j=0,\ldots,p$,
$m_{1,j} = -1$, $j=1,\ldots,p$, and $m_{2,j} = K_j - K_1, j=2,\ldots,p$.)

Then, we set
$$ F^{(l)}_j(x):=\frac{F_{j,-}^{(l-1)}(x)-F_{j,+}^{(l-1)}(x)}{F_{l-1,-}^{(l-1)}(x)-F_{l-1,+}^{(l-1)}(x)},\qquad
x\in\Delta_{l-1},
$$
for $j=l,\ldots,p$. We will see that this function can be analytically extended to $\cee\setminus\Delta_l$, and that it
is again a Cauchy transform
\begin{equation}\label{Nik:hier:0} F^{(l)}_j(x) = \int_{\Delta_l}\frac{1}{x-t}
\ud\nu^{(l)}_j(t),\qquad x\in\cee\setminus\Delta_l,
\end{equation}
for some constant sign measure $\nu^{(l)}_j$ supported on $\Delta_l$. The measures $(\nu^{(l)}_l,\ldots,\nu^{(l)}_p)$
form the $l$th layer of the Nikishin hierarchy. With this we will conclude the induction.

In fact, using the formulas from the induction hypothesis, it follows that
\[F^{(l)}_j(x)= \frac{m_{l-1,j}}{m_{l-1,l-1}}  \frac{\psi^{(l-1)}_0(x)\ldots
\psi^{(l-1)}_{l-1}(x)}{\psi^{(j)}_0(x)\ldots \psi^{(j)}_{l-1}(x)}, \qquad x \in \Delta_{l-1},
\]
which by the properties of the branches $\psi_k^{(l-1)}$ and $\psi_k^{(j)}, k=0,\ldots,l-1,$ can be extended
analytically to all $\mathbb{C} \setminus \Delta_l$. At infinity the numerator takes a finite value whereas the
denominator has a simple pole. Consequently,
\[ F^{(l)}_j(x) \in \mathcal{H}(\mathbb{C}\setminus \Delta_l), \qquad F^{(l)}_j(x) = O\left(\frac{1}{x}\right),\qquad
x
\to \infty.
\]
Therefore, using the Cauchy integral theorem, we obtain
\[ F^{(l)}_j(x) = \frac{1}{2\pi i}\int_{\gamma} \frac{F^{(l)}_j(t)}{x-t} \ud t
\]
for all $x$ exterior to $\gamma$, where $\gamma$ is a positively oriented closed Jordan curve that surrounds
$\Delta_{l}$. Shrinking $\gamma$ to $\Delta_l$, we find that \eqref{Nik:hier:0} takes place with
\begin{equation}\label{constants:cij:nik} \frac{\ud \nu^{(l)}_j(x)}{\ud x} = \frac{1}{2\pi
i}\left(F_{j,-}^{(l)}(x)-F_{j,+}^{(l)}(x)\right) = \frac{m_{l,j}}{2\pi i} \frac{\psi^{(l-1)}_0(x)\ldots
\psi^{(l-1)}_{l-2}(x)}{\psi^{(j)}_0(x)\ldots \psi^{(j)}_l(x)}\left(\psi^{(1)}_{l-1,-}(x)-\psi^{(1)}_{l-1,+}(x)\right),
\end{equation}
for $x \in \Delta_l$, with $m_{l,j} = \frac{m_{l-1,j}(K_{j}-K_{l-1})}{m_{l-1,l-1}}$. Here, we also used
\eqref{relfund}. Using the expression for $\frac{\ud \nu^{(l)}_j(x)}{\ud x}$, arguing as we did above for the case
$l=2$  one deduces that $\nu^{(l)}_j$ has constant sign on $\Delta_l$ for each $j=l,\ldots,p$. The fact that the
measures are absolutely continuous with respect to the Lebesgue measure is a direct consequence of the expression we
have found for them. This ends the proof of Theorem~\ref{theorem:main:Nikishin}. $\bol$

\begin{remark}\label{remark:diagonal:nik} Note that the generating measures $\rho_j$
in Theorem~\ref{theorem:main:Nikishin} are nothing but the ``diagonal" measures in the Nikishin hierarchy, i.e.,
$\rho_j = \nu_j^{(j)}$ given by \eqref{constants:cij:nik}.
\end{remark}

\begin{remark}\label{recursion:clj}
The recursive relation given above between the constants $m_{l,j}$ allows us to write
\begin{equation}\label{formula:clj}
m_{l,j}=\frac{\prod_{i=1}^{l-1} \left(K_{j}-K_{i}\right)}{\prod_{i=1}^{l-2} \left(K_{l-1}-K_{i}\right)}
=\frac{\prod_{i=0}^{l-2} \left(b_{i,i}-b_{j-1,j-1}\right)}{\prod_{i=0}^{l-3} \left(b_{i,i}-b_{l-2,l-2}\right)},
\end{equation}
for all $l=2,\ldots,p$ and $j=l,\ldots,p$, where the second equality follows from \eqref{def:Kj}.
\end{remark}

\section{Proof of Theorem~\ref{theorem:relations:bij}}
\label{section:rellimitcoeff}

Let
\[
F_{k}(z)=z-b_{k-1,k-1}+\sum_{l=1}^{\infty}\frac{c_{l}}{z^{l}},\qquad |z|\geq R_{0}, \qquad k=1,\ldots,p,
\]
denote the Laurent expansion at infinity of $F_{k}$, where we are using \eqref{def:Kj}--\eqref{constants:Ki}. Set
\[ F(z) = z+\sum_{l=1}^{\infty}\frac{c_{l}}{z^{l}},\qquad |z|\geq R_{0}.
\]
Obviously,
\begin{equation} \label{eq:Fs} F_k(z) = F(z) - b_{k-1,k-1}.
\end{equation}
Since $F$ is a translation of $F_k$, which is one branch of a conformal representation of the Riemann surface
$\mathcal{R}$ onto $\overline{\mathbb{C}}$ (see \eqref{ratio:asy:psi}), $F$ itself extends to a conformal
representation of $\mathcal{R}$ onto $\overline{\mathbb{C}}$. We denote by $\mathcal{F}$ the analytic extension of $F$
to all ${\mathcal{R}}$. For $j\neq k$ we have that $b_{j-1,j-1} \neq b_{k-1,k-1}$, otherwise $F_j = F_k$. Taking these
facts into account, we conclude that there exist $p$ distinct points $\zeta_0,\ldots,\zeta_{p-1} \in \mathcal{R}$ such
that
\[ \mathcal{F}(\zeta_k) = b_{k,k}, \qquad k=0,\ldots,p-1.
\]
In the sequel $\zeta_n = \zeta_k$ if $n \equiv k \mod p, k \in \{0,\ldots,p-1\}$.

It follows from \eqref{eq:relFb} and \eqref{eq:Fs} that
\[
b_{k,k-1}=\lim_{z\rightarrow\infty}(z-F (z) ) F_{k}(z)=-c_{1}=:\beta, \qquad k=0,\ldots,p-1,
\]
so \eqref{rel:b:1} is proved.

Using again \eqref{eq:relFb} and \eqref{eq:Fs}, we find that
\begin{equation} \label{eq:fund}
zF_{k}\cdots F_{k-p+1} = F F_{k} \cdots F_{k-p+1} + b_{k,k-1}F_{k-1}\cdots F_{k-p+1} + \cdots + b_{k,k-p}.
\end{equation}
Writing equation \eqref{eq:fund} substituting $k$ by $k+1$ and taking into consideration that $F_{k}\cdots
F_{k-p+1}=F_{k+1}\cdots F_{k-p+2}$, we obtain the equality
\[ b_{k,k-1}F_{k-1}\cdots F_{k-p+1} + \cdots + b_{k,k-p} = b_{k+1,k}F_{k}\cdots F_{k-p+2} + \cdots + b_{k+1,k-p+1},
\]
which due to \eqref{eq:Fs}, and taking into consideration that $F$ extends to all $\mathcal{R}$, takes the form
\begin{equation} \label{eq:fund1}  b_{k,k-1}(\mathcal{F} - b_{k-2,k-2})\cdots (\mathcal{F} - b_{k-p,k-p}) + \cdots +
b_{k,k-p+1}(\mathcal{F} - b_{k-p,k-p}) + b_{k,k-p} =
\end{equation}
\[
b_{k+1,k}(\mathcal{F} - b_{k-1,k-1})\cdots (\mathcal{F} - b_{k-p+1,k-p+1}) + \cdots + b_{k+1,k-p+2}(\mathcal{F} -
b_{k-p+1,k-p+1}) + b_{k+1,k-p+1}.
\]
Evaluating  \eqref{eq:fund1} at $\zeta_{k-p+1} \in \mathcal{R}$, it follows that
\begin{equation} \label{eq:aux1} b_{k,k-p+1}(b_{k-p+1,k-p+1}- b_{k-p,k-p}) + b_{k,k-p} = b_{k+1,k-p+1},
\end{equation}
which is equivalent to \eqref{rel:b:2} when $l = p-1$.

Now, evaluating \eqref{eq:fund1} at $\zeta_{k-p+2} \in \mathcal{R}$, we obtain
\[ b_{k,k-p+2}(b_{k-p+2,k-p+2}  - b_{k-p+1,k-p+1})(b_{k-p+2,k-p+2}  - b_{k-p,k-p}) +
\]
\[ b_{k,k-p+1}(b_{k-p+2,k-p+2}  - b_{k-p,k-p}) + b_{k,k-p}=
\]
\[ b_{k,k-p+2}(b_{k-p+2,k-p+2}  - b_{k-p+1,k-p+1})(b_{k-p+2,k-p+2}  - b_{k-p,k-p}) +
\]
\[b_{k,k-p+1}(b_{k-p+2,k-p+2}  - b_{k-p+1,k-p+1}) +
\]
\[ b_{k,k-p+1}(b_{k-p+1,k-p+1}  - b_{k-p,k-p})+ b_{k,k-p}=
\]
\[ b_{k+1,k-p+2}(b_{k-p+2,k-p+2} - b_{k-p+1,k-p+1}) + b_{k+1,k-p+1}.
\]
Using  \eqref{eq:aux1}, this equality reduces to
\[ b_{k,k-p+2}(b_{k-p+2,k-p+2}  - b_{k-p+1,k-p+1})(b_{k-p+2,k-p+2}  - b_{k-p,k-p}) +
\]
\[ b_{k,k-p+1}(b_{k-p+2,k-p+2}  - b_{k-p+1,k-p+1}) =
 b_{k+1,k-p+2}(b_{k-p+2,k-p+2} - b_{k-p+1,k-p+1}).
\]
We can cancel out $b_{k-p+2,k-p+2} - b_{k-p+1,k-p+1} (\neq 0)$ from both sides and we obtain
\begin{equation} \label{eq:aux2} b_{k,k-p+2} (b_{k-p+2,k-p+2}  - b_{k-p,k-p}) +  b_{k,k-p+1}  = b_{k+1,k-p+2}
\end{equation}
which is \eqref{rel:b:2} for $l=p-2$.

In general, assume that \eqref{rel:b:2} is valid for all $l\in\{p-1,p-2,\ldots,p-m\}$, with $m\leq p-2$. In order to
prove that it is also valid for $\widetilde{l}:=p-m-1$, we evaluate \eqref{eq:fund1} at $\zeta_{k-\widetilde{l}}$, and
we get
\[
b_{k,k-\widetilde{l}}(b_{k-\widetilde{l},k-\widetilde{l}}-b_{k-\widetilde{l}-1,k-\widetilde{l}-1})
\ldots(b_{k-\widetilde{l},k-\widetilde{l}}-b_{k-p,k-p})+
\]
\[
b_{k,k-\widetilde{l}}(b_{k-\widetilde{l},k-\widetilde{l}}-b_{k-\widetilde{l}-2,k-\widetilde{l}-2})
\ldots(b_{k-\widetilde{l},k-\widetilde{l}}-b_{k-p,k-p})+
\]
\[
\ldots+b_{k,k-p+2}(b_{k-\widetilde{l},k-\widetilde{l}}-b_{k-p+1,k-p+1})
(b_{k-\widetilde{l},k-\widetilde{l}}-b_{k-p,k-p})+
\]
\[
b_{k,k-p+1}(b_{k-\widetilde{l},k-\widetilde{l}}-b_{k-p,k-p})+b_{k,k-p}=
\]
\[
b_{k+1,k-\widetilde{l}}(b_{k-\widetilde{l},k-\widetilde{l}}-b_{k-\widetilde{l}-1,k-\widetilde{l}-1})
\ldots(b_{k-\widetilde{l},k-\widetilde{l}}-b_{k-p+1,k-p+1})+
\]
\[
b_{k+1,k-\widetilde{l}-1}(b_{k-\widetilde{l},k-\widetilde{l}}-b_{k-\widetilde{l}-2,k-\widetilde{l}-2})
\ldots(b_{k-\widetilde{l},k-\widetilde{l}}-b_{k-p+1,k-p+1})+
\]
\[
\ldots+b_{k+1,k-p+3}(b_{k-\widetilde{l},k-\widetilde{l}}-b_{k-p+2,k-p+2})
(b_{k-\widetilde{l},k-\widetilde{l}}-b_{k-p+1,k-p+1})+
\]
\[
b_{k+1,k-p+2}(b_{k-\widetilde{l},k-\widetilde{l}}-b_{k-p+1,k-p+1})+b_{k+1,k-p+1}.
\]
On the left-hand side we replace $b_{k,k-p+1}(b_{k-\widetilde{l},k-\widetilde{l}}-b_{k-p,k-p})$ by
\begin{equation}\label{replacement1}
b_{k,k-p+1}(b_{k-\widetilde{l},k-\widetilde{l}}-b_{k-p+1,k-p+1}) +b_{k,k-p+1}(b_{k-p+1,k-p+1}-b_{k-p,k-p}).
\end{equation}
Now we apply \eqref{rel:b:2} for $l=p-1$, and this allows us to delete the second term in \eqref{replacement1} along
with $b_{k,k-p}$ on the left-hand side and $b_{k+1,k-p+1}$ on the right-hand side. Now the factor
$b_{k-\widetilde{l},k-\widetilde{l}}-b_{k-p+1,k-p+1}$ appears in all the terms on both sides of the resulting equation.
We cancel out this factor everywhere to obtain a new equation in which on the left-hand side we have the expression
$$
b_{k,k-p+2}(b_{k-\widetilde{l},k-\widetilde{l}}-b_{k-p,k-p})+b_{k,k-p+1},
$$
and on the right-hand side we have
$$
b_{k+1,k-p+3}(b_{k-\widetilde{l},k-\widetilde{l}}-b_{k-p+2,k-p+2})+b_{k+1,k-p+2}.
$$
In a similar way, we now replace $b_{k,k-p+2}(b_{k-\widetilde{l},k-\widetilde{l}}-b_{k-p,k-p})$ by
\begin{equation}\label{replacement2}
b_{k,k-p+2}(b_{k-\widetilde{l},k-\widetilde{l}}-b_{k-p+2,k-p+2}) +b_{k,k-p+2}(b_{k-p+2,k-p+2}-b_{k-p,k-p})
\end{equation}
and apply \eqref{rel:b:2} for $l=p-2$. This allows us to delete the second term in \eqref{replacement2} along with
$b_{k,k-p+1}$ on the left-hand side and $b_{k+1,k-p+2}$ on the right-hand side. The factor
$b_{k-\widetilde{l},k-\widetilde{l}}-b_{k-p+2,k-p+2}$ appears in all the terms on both sides of the resulting equation.
We cancel out this factor.

It is clear that continuing in this fashion we will arrive at the desired equation. Relation \eqref{rel:b:1} may be
regarded a special case of \eqref{rel:b:2} for $l=0$. We are done. $\bol$

\begin{question} Suppose that a collection of real numbers $\{b_{k,j}\}$ is given satisfying \eqref{rel:b:1}--\eqref{rel:b:2}. What other conditions must be added in order that the given collection of numbers corresponds to the periodic coefficients of the recurrence relation satisfied by a sequence of Chebyshev-Nikishin polynomials?
\end{question}

\section{Formulas and strong asymptotics for the second kind functions}
\label{section:Widom:secondkind}

In this section we obtain strong asymptotics and an exact Widom-type formula for the second kind functions $\Psi_{n,l}$
in \eqref{secondkind:1}--\eqref{secondkind:2}. The function $\Psi_{n,l}$ (with $1 \leq l\leq p$) satisfies the
following boundary value problem:

\begin{bvp}\label{bvp:Psi:nk}\textrm{}
\begin{enumerate}
\item[$(a)$] $\Psi_{n,l}$ is analytic for $z\in\cee\setminus\Delta_l$. \item[$(b)$] $\Psi_{n,l}$ satisfies the jump
    relation ($\Psi_{n,0}(z) = P_n(z)$)
\begin{equation}\label{jump:secondkind}
(\Psi_{n,l}(x))_- - (\Psi_{n,l}(x))_+ = 2\pi i\,\Psi_{n,l-1}(x)\frac{\ud \rho_l(x)}{\ud x},\qquad x\in\Delta_l.
\end{equation}
$$
\Psi_{n,l}(z)=O\Big(\frac{1}{z^{n_{l}+1}}\Big),\qquad z\to\infty,
$$
with $n_l$ in \eqref{multi:index}. \item[$(d)$] $\Psi_{n,l}$ stays bounded near the endpoints $\alpha_l$ and
$\beta_l$ of the interval $\Delta_l$, i.e., $\Psi_{n,l}(z)=O(1)$ as $z\to\alpha_l$ and $z\to\beta_{l}$.
\end{enumerate}
\end{bvp}

Properties $(a)$ and $(b)$ follow immediately from the definition of $\Psi_{n,l}$. Property $(c)$ is a consequence of
the orthogonality properties satisfied by $\Psi_{n,l-1}$ with respect to the measure $\mathrm{d} \rho_{l}$, proved in
\cite[Proposition 1]{GRS}. Finally, $(d)$ follows from the fact that $\Psi_{n,l-1}$ is analytic on $\Delta_{l}$ and
\[
\frac{\mathrm{d} \rho_{l}(x)}{\mathrm{d} x} =\frac{\mathrm{d} \nu_{l}^{(l)}(x)}{\mathrm{d} x}= \left\{
\begin{array}{cc}
O(|x-\alpha_{l}|^{1/2}), & x\rightarrow\alpha_{l},\\[0.5em] O(|x-\beta_{l}|^{1/2}), & x\rightarrow\beta_{l},
\end{array}
\right.
\]
see \eqref{constants:cij:nik} and Remark~\ref{remark:diagonal:nik}. Given $\Psi_{n,l-1}$ (see \eqref{jump:secondkind}),
it is readily seen that  $\Psi_{n,l}(z)$ is uniquely determined by the above boundary value problem.

\begin{proposition}\label{PropSecondKind1}
Fix $l\in\{0,\ldots,p\}$. For any $m$ large enough and for any $k=0,\ldots,p-1$, the following Widom-type formula
holds:
\begin{multline}\label{Widom:psi:n2}
\Psi_{mp+k,l}(x)=\\ \frac{(-1)^{p+k}}{f_{p}}\sum_{j=l+1}^{p+1} \frac{\det F^{[p,k+1]}(z_j(x),x)}{\prod_{t=1,t\neq
j}^{p+1}(z_{j}(x)-z_{t}(x))}\left(
\prod_{i=0}^{l-1}m_{i+1,i+1}\frac{\psi^{(1)}_{i}(x)-\psi^{(1)}_{j-1}(x)}{\psi^{(l)}_{i}(x)
\psi^{(i+1)}_{j-1}(x)}\right)z_j(x)^{-m-1}.
\end{multline}
The constants $m_{i+1,i+1}$ are given in \eqref{constants:cij:nik}--\eqref{formula:clj}.
\end{proposition}

\begin{remark}
Formula \eqref{Widom:psi:n2} can be computed for all $x\in\cee\setminus\mathcal A$, where $\mathcal A$ consists of
those points $x\in\cee$ for which $z_i(x)=z_j(x)$ for certain $i\neq j$. Thus $\mathcal A$ is a finite set and it is
formed by the zeros of the discriminant of $f(z,x)=0$ in \eqref{algebraic:equation} with respect to the variable $z$.
Note that the endpoints of the intervals $\Delta_{l}$ are points in $\mathcal{A}$.

The product of the entries $m_{i+1,i+1}$ in \eqref{Widom:psi:n2} can be evaluated using \eqref{formula:clj}:
\[
\prod_{i=0}^{l-1}m_{i+1,i+1}=-\prod_{i=1}^{l-1}(K_{l}-K_{i})=-\prod_{i=0}^{l-2}(b_{i,i}-b_{l-1,l-1}).
\]
(The minus sign comes from $m_{1,1}=-1$.) It also follows from \eqref{ratio:asy:PQ}, \eqref{strongasy:Pn}, and
\eqref{ratio:asy:psi} that
\begin{equation}\label{minors:psi:1}
\psi^{(k)}_{j-1}(x) = -\frac{\det F^{[p,k+1]}(z_j(x),x)}{\det F^{[p,k]}(z_j(x),x)},
\end{equation}
for all $k=1,\ldots,p-1$ and $j=1,\ldots,p+1$, and
\begin{equation}\label{minors:psi:2} \psi^{(p)}_{j-1}(x) =(-1)^{p-1}
z_j^{-1}(x)\frac{\det F^{[p,1]}(z_j(x),x)}{\det F^{[p,p]}(z_j(x),x)},
\end{equation}
for all $j=1,\ldots,p+1$, so in fact all the quantities in \eqref{Widom:psi:n2} depend on the minors of the block
Toeplitz symbol $F(z,x)$ \eqref{symbol:Toeplitz} evaluated at the roots $z=z_j(x)$.

We will prove formula \eqref{Widom:psi:n2} only for $m$ large enough, although it may actually be valid for all $m$.
\end{remark}

\begin{proof}[Proof of Proposition~\ref{PropSecondKind1}]
We prove \eqref{Widom:psi:n2} using induction on $l$. The case $l=0$ reduces to \eqref{Widom:Pn}. Let us assume as
induction hypothesis that \eqref{Widom:psi:n2} holds with $l$ replaced by $l-1$. We need to check that the right hand
side of \eqref{Widom:psi:n2} satisfies the four conditions $(a)$--$(d)$ in the boundary value problem~\ref{bvp:Psi:nk}.
Actually we will prove that the right hand side of \eqref{Widom:psi:n2}, which we denote by
$\widetilde{\Psi}_{mp+k,l}$, satisfies a weaker version of $(a)$--$(d)$ where these conditions are replaced by
\begin{itemize}
\item[$(a')$] $\widetilde{\Psi}_{mp+k,l}(z)$ is analytic for $z\in\cee\setminus(\Delta_l\cup\mathcal A)$.
    \item[$(b')$] $\widetilde{\Psi}_{mp+k,l}$ satisfies the jump relation
\[
(\widetilde{\Psi}_{mp+k,l}(x))_- - (\widetilde{\Psi}_{mp+k,l}(x))_+ = 2\pi
i\,\widetilde{\Psi}_{mp+k,l-1}(x)\frac{\ud \rho_l(x)}{\ud x},\qquad x\in\Delta_l\setminus\mathcal{A}.
\]
\item[$(c')$] As $z\rightarrow\infty$,
\[
\widetilde{\Psi}_{mp+k,l}(z)=O\Big(\frac{1}{z^{m+\delta}}\Big),
\]
where $\delta\in\mathbb{R}$ is some fixed constant independent of $m$. \item[$(d')$] Near each point
$a_i\in\mathcal A$ there is a fixed integer $q_i$ such that $\widetilde{\Psi}_{mp+k,l}(z)=O((z-a_i)^{-q_i/2})$.
\end{itemize}
The integer $q_i$ will be precisely the multiplicity of $x=a_i\in\mathcal A$ as a zero of the discriminant of
$f(z,x)=0$. The solution to the boundary value problem $(a')$--$(d')$ is unique only for $m$ large enough, since then
the (possible) poles at the points in $\mathcal A$ cannot compete with the high-order zero at infinity. This will prove
that $\Psi_{mp+k,l}\equiv \widetilde{\Psi}_{mp+k,l}$; that is, \eqref{Widom:psi:n2} holds.

In view of \eqref{minors:psi:1}--\eqref{minors:psi:2}, the function $\widetilde{\Psi}_{mp+k,l}$ is a symmetric function
of $z_{l+1},\ldots,z_{p+1}$, and it is also a symmetric function of the variables $z_{1},\ldots,z_{l}$. More
specifically, if we view $\widetilde{\Psi}_{mp+k,l}$ as a function
$\Lambda(z_{1},\ldots,z_{l},z_{l+1},\ldots,z_{p+1},x)$ in the variables $z_{k}$ and $x$, then
\[
\Lambda(z_{1},\ldots,z_{l},z_{l+1},\ldots,z_{p+1},x)=\Lambda(z_{1}^{*},\ldots,z_{l}^{*},z_{l+1}^{*},\ldots,z_{p+1}^{*},x),
\]
where $(z_{1}^{*},\ldots,z_{l}^{*})$ is any permutation of $(z_{1},\ldots,z_{l})$, and
$(z_{l+1}^{*},\ldots,z_{p+1}^{*})$ is any permutation of $(z_{l+1},\ldots,z_{p+1})$. This property readily implies the
analyticity of $\widetilde{\Psi}_{mp+k,l}$ on $\cee\setminus(\Delta_l\cup\mathcal{A})$, which is Part $(a')$.

For Part $(b')$ we distinguish in the expression of $\widetilde{\Psi}_{mp+k,l}(x)$ the term with $j=l+1$ from the terms
with $j=l+2,\ldots,p+1$. We start with the latter terms. First we write the factor with $i=l-1$ in the product in
\eqref{Widom:psi:n2} as
\begin{equation}\label{widom:split}
\frac{\psi^{(1)}_{l-1}(x)-\psi^{(1)}_{j-1}(x)}{\psi^{(l)}_{l-1}(x)\psi^{(l)}_{j-1}(x)} =
\frac{\psi^{(l)}_{l-1}(x)-\psi^{(l)}_{j-1}(x)}{\psi^{(l)}_{l-1}(x)\psi^{(l)}_{j-1}(x)}
=\frac{1}{\psi^{(l)}_{j-1}(x)}-\frac{1}{\psi^{(l)}_{l-1}(x)}.
\end{equation}
Observe now that the $+$ and $-$ boundary values on $\Delta_{l}\setminus\mathcal{A}$ of the expression
\begin{equation}\label{eq:aux4}
\frac{(-1)^{p+k}}{f_{p}}\sum_{j=l+2}^{p+1} \frac{\det F^{[p,k+1]}(z_j(x),x)}{\prod_{t=1,t\neq
j}^{p+1}(z_{j}(x)-z_{t}(x))}
\left(\prod_{i=0}^{l-2}m_{i+1,i+1}\frac{\psi^{(1)}_{i}(x)-\psi^{(1)}_{j-1}(x)}{\psi^{(l)}_{i}(x)
\psi^{(i+1)}_{j-1}(x)}\right)\frac{z_j(x)^{-m-1}}{\psi_{j-1}^{(l)}(x)}
\end{equation}
are equal. Indeed, this follows since, with the identifications \eqref{minors:psi:1}--\eqref{minors:psi:2}, equation
\eqref{eq:aux4} depends symmetrically on the variables $z_{l+2},\ldots,z_{p+1}$, depends symmetrically also on
$z_1,\ldots,z_{l-1}$, and does not depend on $z_l$ or $z_{l+1}$.

The equality of the $+$ and $-$ boundary values of \eqref{eq:aux4} also holds if we drop the factor
$1/\psi_{j-1}^{(l)}(x)$. In view of \eqref{widom:split}, it then follows that if we subtract the $+$ boundary value
from the $-$ boundary value of the terms in $\widetilde{\Psi}_{mp+k,l}(x)$ with $j=l+2,\ldots,p+1$, we obtain
\begin{multline}\label{eq:aux5}
m_{l,l}\frac{\psi_{l-1,-}^{(l)}(x) -\psi_{l-1,+}^{(l)}(x)}{\psi_{l-1}^{(l)}(x)\psi_{l}^{(l)}(x)}\\
\times\frac{(-1)^{p+k}}{f_{p}}\sum_{j=l+2}^{p+1} \frac{\det F^{[p,k+1]}(z_j(x),x)}{\prod_{t=1,t\neq
j}^{p+1}(z_{j}(x)-z_{t}(x))}\,
\left(\prod_{i=0}^{l-2}m_{i+1,i+1}\frac{\psi^{(1)}_{i}(x)-\psi^{(1)}_{j-1}(x)}{\psi^{(l)}_{i}(x)
\psi^{(i+1)}_{j-1}(x)}\right)z_j(x)^{-m-1}.
\end{multline}
Here we also used that the product $\psi_{l-1,-}^{(l)}(x)\psi_{l-1,+}^{(l)}(x)=\psi_{l-1}^{(l)}(x)\psi_{l}^{(l)}(x)$ is
analytic across $\Delta_l$.

Using the relations $z_{l+1,\pm}(x)=z_{l,\mp}(x)$, $\psi_{l,\pm}^{(s)}(x)=\psi_{l-1,\mp}^{(s)}(x)$, $x\in\Delta_{l}$,
and other simple considerations, it is easy to see that the difference of the $-$ and $+$ boundary values of the term
in $\widetilde{\Psi}_{mp+k,l}(x)$ corresponding to $j=l+1$ is given by the expression in \eqref{eq:aux5} with the sum
$\sum_{j=l+2}^{p+1}$ replaced by $\sum_{j=l}^{l+1}$. In conclusion,
\begin{multline}\label{eq:aux6}
(\widetilde{\Psi}_{mp+k,l}(x))_{-} - (\widetilde{\Psi}_{mp+k,l}(x))_{+}=m_{l,l}\frac{\psi_{l-1,-}^{(l)}(x)
-\psi_{l-1,+}^{(l)}(x)}{\psi_{l-1}^{(l)}(x)\psi_{l}^{(l)}(x)}\\ \times\frac{(-1)^{p+k}}{f_{p}}\sum_{j=l}^{p+1}
\frac{\det F^{[p,k+1]}(z_j(x),x)}{\prod_{t=1,t\neq j}^{p+1}(z_{j}(x)-z_{t}(x))}
\left(\prod_{i=0}^{l-2}m_{i+1,i+1}\frac{\psi^{(1)}_{i}(x)-\psi^{(1)}_{j-1}(x)}{\psi^{(l)}_{i}(x)
\psi^{(i+1)}_{j-1}(x)}\right)z_j(x)^{-m-1}.
\end{multline}
By the induction hypothesis, we know that $\Psi_{mp+k,l-1}(x)$ is given by the expression in the last line of
\eqref{eq:aux6} with $\psi_{i}^{(l)}(x)$ replaced by $\psi_{i}^{(l-1)}(x)$. This observation and
\eqref{constants:cij:nik} gives Part $(b')$.

Finally, the conditions $(c')$ (cf. \eqref{growth:inf}) and $(d')$ are obvious. This proves \eqref{Widom:psi:n2}.
\end{proof}

\begin{remark}
The functions defined by
\[
\Phi_{n,l}(z):=\int_{\Delta_{1}}\frac{P_{n}(t)}{z-t}\ud\nu_{l}(t), \qquad l=1,\ldots,p,
\]
where the measures $\nu_{l}$ are the orthogonality measures in Theorem \ref{theorem:main:ortho}, are also often called
second kind functions (see \cite[pg. 672]{GRS}), and represent the linear remainders in the Hermite-Pad\'{e}
approximation to the functions $\int_{\Delta_{1}}\frac{\ud\nu_{l}(t)}{z-t}$, $l=1,\ldots,p$. It is easy to see that for
all $m$ large enough, we have the formula
\begin{equation}\label{formula:Phinl}
\Phi_{mp+k,l}(x)=\frac{(-1)^{p+k}}{f_{p}}\sum_{j=2}^{p+1} \frac{\det F^{[p,k+1]}(z_j(x),x)}{\prod_{t=1,t\neq
j}^{p+1}(z_{j}(x)-z_{t}(x))}\, \left(\frac{1}{\psi_{0}^{(l)}(x)}-\frac{1}{\psi_{j-1}^{(l)}(x)}\right) z_j(x)^{-m-1}.
\end{equation}
We leave the justification of \eqref{formula:Phinl} to the reader.
\end{remark}

\begin{corollary}
The following strong asymptotic formulas hold uniformly on compact subsets of the indicated regions. For each fixed
$k\in\{0,\ldots,p-1\}$ and $l\in\{1,\ldots,p\}$,
\begin{multline*}
\lim_{m\rightarrow\infty}\Psi_{mp+k,l}(x)\,z_{l+1}(x)^{m+1}\\ =\frac{(-1)^{p+k}}{f_{p}}\frac{\det
F^{[p,k+1]}(z_{l+1}(x),x)}{\prod_{t=1,t\neq l+1}^{p+1}(z_{l+1}(x)-z_{t}(x))}\,
\left(\prod_{i=0}^{l-1}m_{i+1,i+1}\frac{\psi^{(1)}_{i}(x)-\psi^{(1)}_{l}(x)}{\psi^{(l)}_{i}(x)
\,\psi^{(i+1)}_{l}(x)}\right),\quad x\in\mathbb{C}\setminus(\Delta_{l}\cup\Delta_{l+1}),
\end{multline*}
and
\begin{multline*}
\lim_{m\rightarrow\infty}\Phi_{mp+k,l}(x)\,z_{2}(x)^{m+1}\\ =\frac{(-1)^{p+k}}{f_{p}}\frac{\det
F^{[p,k+1]}(z_{2}(x),x)}{\prod_{t=1,t\neq 2}^{p+1}(z_{2}(x)-z_{t}(x))}\,
\left(\frac{1}{\psi_{0}^{(l)}(x)}-\frac{1}{\psi_{1}^{(l)}(x)}\right) ,\qquad
x\in\mathbb{C}\setminus(\Delta_{1}\cup\Delta_{2}).
\end{multline*}
\end{corollary}
\begin{proof}
This is an immediate consequence of \eqref{Widom:psi:n2}, \eqref{formula:Phinl}, and the fact that
$|z_{l+1}(x)|<|z_{j}(x)|$ for all $j\geq l+2$ and $x\in\mathbb{C}\setminus(\Delta_{l}\cup\Delta_{l+1})$ (cf.
\eqref{ordering:roots} and \eqref{defGammas}).
\end{proof}

\section{Alternative systems of multi-indices}
\label{section:concluding:remark}

The results in this paper were formulated for the standard ``staircase" system of multi-indices $\mathbf n$ in
\eqref{multi:index}. More generally, let $\Pi:=(\pi_1,\ldots,\pi_p)$ be an arbitrary permutation of $(1,\ldots,p)$. For
each $n\in\zet_{\geq 0}$, define
$$\mathbf{n}^{\Pi}:=(n_1^{\Pi},\ldots,n_p^{\Pi})\in\zet_{\geq 0}^p$$ as the
unique multi-index such that
\begin{equation}\label{multi:index:alt}n_{\pi_1}^{\Pi}\geq \ldots\geq
n_{\pi_p}^{\Pi}\geq n_{\pi_1}^{\Pi}-1,\qquad \textrm{and}\ \ \ |\mathbf{n}^{\Pi}|:=
n_{1}^{\Pi}+\ldots+n_{p}^{\Pi}=n.\end{equation} For the trivial permutation $\Pi=(\pi_1,\ldots,\pi_p)=(1,\ldots,p)$,
this definition reduces to \eqref{multi:index}.

Let $Q_n^{\Pi}(z)$ be the monic multiple orthogonal polynomial of degree $n$ satisfying \eqref{Qn:ortho}, with $n_k$
replaced by $n_k^{\Pi}$. The polynomials $Q_n^{\Pi}(z)$ satisfy a recurrence relation of the form \eqref{recrelQn},
with recurrence coefficients $a^{\Pi}_{n,m}$ having periodic limits $b^{\Pi}_{i,j}$ as in \eqref{def:limitreccoeff}. We
define the block Toeplitz matrix $T^{\Pi}$ and the Chebyshev-Nikishin polynomials $P^{\Pi}_n(z)$ as in
\eqref{defH}--\eqref{blocks:Bminus1}  and \eqref{Pn:def} respectively, with $b^{\Pi}_{i,j}$ instead of $b_{i,j}$.

We have the following generalization to Theorem~\ref{theorem:relations:bij}: The numbers $b^{\Pi}_{i,j}$ satisfy the
relations \eqref{rel:b:1}--\eqref{rel:b:2}, with $\beta$ independent of the permutation $\Pi$.

We also have the following generalization to Theorem~\ref{theorem:main:ortho}.

\begin{theorem}\label{theorem:multi}
The Chebyshev-Nikishin polynomials $P_n^{\Pi}$ satisfy the orthogonality conditions
\begin{equation}\label{Pn:mop:alt}
\int_{\Delta_1} P_{n}^{\Pi}(x)\, x^{l} \ud\nu_k(x)=0,\qquad l=0,\ldots,n_{k}^{\Pi}-1,\quad k=1,\ldots,p,
\end{equation}
with $\nu_k$ the measure defined by \eqref{cauchy1} (independently of the permutation $\Pi$), and with $n_{k}^{\Pi}$
the $k$-th component of $\mathbf{n}^{\Pi}$ in \eqref{multi:index:alt}.
\end{theorem}

Theorem~\ref{theorem:multi} can be proved in a similar way as Theorem~\ref{theorem:main:ortho} and we leave the details
to the interested reader. See also the next paragraphs for more information.

We define $F^{\Pi}_{k}(z)$ as in \eqref{ratioasymp}, with each $Q_n$ replaced by $Q^{\Pi}_n$. The main result in
\cite{AptLopRoc} implies that $F^{\Pi}_{k}(z)=F_{\pi_k}(z)$.

Define the quantities $F^{\Pi}(z,x)$, $f^{\Pi}(z,x)$, $z_k^{\Pi}(x)$ and $\Gamma_k^{\Pi}$ as in
\eqref{symbol:Toeplitz}--\eqref{defGammas}, with $b^{\Pi}_{i,j}$ instead of $b_{i,j}$. Reproducing the proof in
Section~\ref{section:ratioasy}, we are led to the following analogue of \eqref{ztilde:product}:
\begin{equation*}z_1^{\Pi}(x) = \frac{1}{F^{\Pi}_{1}(x)\ldots F^{\Pi}_{p}(x)} = \frac{1}{F_{1}(x)\ldots
F_{p}(x)}=z_1(x),
\end{equation*}
for all $x$ sufficiently large, where we used the above observation that $F^{\Pi}_{k}(z)=F_{\pi_k}(z)$. So for all $x$
sufficiently large, $z_1^{\Pi}(x)=z_1(x)$  is independent of the permutation $\Pi$. Analytic continuation then implies
that each of the roots $z_k^{\Pi}(x)=z_k(x)$ is independent of $\Pi$. This implies in turn that the algebraic equation
$f^{\Pi}(z,x)=f(z,x)=0$ is independent of $\Pi$ and therefore also $\Gamma_k^{\Pi}=\Gamma_k=\Delta_k$.

\begin{question} Given the intervals
$\Delta_1,\ldots,\Delta_p$, we have now obtained $p!$ block Toeplitz matrices $T^{\Pi}$, labeled by the permutations
$\Pi$ of $(1,\ldots,p)$, which all have the same algebraic equation $f(z,x)=0$, and which all satisfy
\eqref{GammaDelta}. Are there any other block Toeplitz matrices $T$ of the form \eqref{defH}--\eqref{blocks:Bminus1}
(with arbitrary entries $b_{i,j}$) leading to this same algebraic equation $f(z,x)=0$? (or equivalently, for which
$\Gamma_k=\Delta_k$ for all $k$, using the notations in \eqref{symbol:Toeplitz}--\eqref{defGammas}?)
\end{question}

We conjecture that the answer to the above question is negative.

\end{document}